\documentclass[12pt]{amsart}
\usepackage{amsmath,amsfonts,euscript,amscd,amsthm,amssymb,upref,graphics}
\theoremstyle{plain}
\newtheorem{theorem}{Theorem}
\swapnumbers

\newtheorem{proposition}[subsection]{Proposition}
\newtheorem{lemma}[subsection]{Lemma}
\newtheorem{corollary}[subsection]{Corollary}

\newtheorem{sit}[subsection]{ }

\theoremstyle{definition}
\newtheorem{definition}[subsection]{Definition}
\newtheorem{example}[subsection]{Example}
\newtheorem{remark}[subsection]{Remark}

\newtheorem{nothing*}[subsection]{}
\newtheorem{notation}[subsection]{Notation}
\newtheorem{convention}[subsection]{Convention}

\newcommand{\rien}[1]{}

\newcommand{\diver}{ \operatorname{{\rm div}}}
\newcommand{\Lie}{ \operatorname{{\rm Lie}}}
\newcommand{\AVF}{ \operatorname{{\rm AVF}}}

\newcommand{\LieA}{ \operatorname{{\rm Lie}_{alg}}}

\newcommand{ \LieAO}{ \operatorname{{\rm Lie}_{alg}^\omega}}

\newcommand{\IVF}{ \operatorname{{\rm IVF}}}

\newcommand{\Aut}{ \operatorname{{\rm Aut}}}

\newcommand{\C}{\ensuremath{\mathbb{C}}}
\newcommand{\Z}{\ensuremath{\mathbb{Z}}}

\newcommand{\R}{\ensuremath{\mathbb{R}}}

\newcommand{\hX}{{\hat X}}
\newcommand{\hR}{{\hat R}}
\newcommand{\hS}{{\hat S}}

\newcommand{\sgoth}{{\ensuremath{\mathfrak{s}}}}
\newcommand{\lgoth}{{\ensuremath{\mathfrak{l}}}}

\newcommand{\mgoth}{{\ensuremath{\mathfrak{m}}}}

\newcommand{\ngoth}{{\ensuremath{\mathfrak{n}}}}

\newcommand{\cB}{{\ensuremath{\mathcal{B}}}}

\newcommand{\cV}{{\ensuremath{\mathcal{V}}}}

\newcommand{\cC}{{\ensuremath{\mathcal{C}}}}

\newcommand{\cN}{{\ensuremath{\mathcal{N}}}}
\newcommand{\cR}{{\ensuremath{\mathcal{R}}}}

\newcommand{\cZ}{{\ensuremath{\mathcal{Z}}}}

\newcommand{\Ker}{{\rm Ker} \,}

\renewcommand{\epsilon}{\varepsilon}
\renewcommand{\phi}{\varphi}

\begin{document}
\renewcommand{\baselinestretch}{1.07}
%%%%%%  TOPMATTER:   %%%%%%%%%%%%%%%%%%%%%%%%%

\title[On Algebraic Volume Density Property]
{On Algebraic Volume Density Property}
\author{Shulim Kaliman}
\address{Department of Mathematics\\
University of Miami\\
Coral Gables, FL 33124 \ \ USA}
%\thanks{Research supported by the NSA grant  H98230-06-1-0063.}
\email{kaliman@math.miami.edu}
\author{Frank Kutzschebauch}
\address{Mathematisches Institut \\Universit\"at Bern
 \\Sidlerstr. 5
 \\ CH-3012 Bern, Switzerland}
\thanks{{\bf Acknowledgements:} This research was started during a visit
of the first author to the University of Bern  and continued
during a visit of the second author to the University of Miami,
Coral Gables. We thank these institutions for their generous
support and excellent working conditions. The research of the
first author was also partially supported by NSA Grant no.
H982301010185 and the
second author was also partially supported by Schweizerische
Nationalfonds grants No. 200020-134876/1 and 200021-140235/1}
\email{Frank.Kutzschebauch@math.unibe.ch} 
{\renewcommand{\thefootnote}{} }
\begin{abstract} A smooth affine algebraic variety $X$ equipped with an algebraic volume form $\omega$ has the algebraic
volume density property (AVDP) if the
Lie algebra generated by complete algebraic vector fields of $\omega$-divergence zero
coincides with the space of all algebraic vector fields of $\omega$-divergence zero. We develop an effective criterion of
verifying whether a given $X$ has AVDP. As an application of this method we establish AVDP for any
homogeneous space $X=G/R$ that admits a $G$-invariant algebraic volume form where
$G$ is a linear algebraic group and $R$ is a closed reductive subgroup of $G$.  \end{abstract}
\maketitle \vfuzz=2pt

\vfuzz=2pt
%%%%%%%%%%%%%%%%%%%%%%%%%%%%%%%%%%%%%%%%%%%%%%%%%%%%%%%%%%%%%%%%%%%
%%%%%%%%%%%%%%%%%%%%%%%%%%%%%%%%%%%%%%%%%%%%%%%%%%%%%%%%%%%%%%%%%%%
%%%%%%%%%%%%%%%%%%%%%%%%%%%%%%%%%%%%%%%%%%%%%%%%%%%%%%%%%%%%%%%%%%%
\section{Introduction} In 1990's  Anders\'en and Lempert  \cite{A}, \cite{AL} discovered a remarkable property of complex Euclidean spaces
 of dimension at least 2  that to a great extend
compensates for the lack of partition of unity for holomorphic automorphisms. It is called the density property  (this terminology
was introduced later by Varolin \cite{V1}) or for short DP. A Stein manifold $X$ has DP if the Lie algebra generated by complete holomorphic
vector fields is dense (in the compact-open topology) in the space of all holomorphic vector fields on $X$. In the presence of DP one can
construct global holomorphic automorphisms of $X$ with prescribed local properties. More precisely, any local phase flow on a Runge domain
in $X$ can be approximated by global automorphisms. Needless to say that this lead to remarkable consequences (see \cite{FR},
\cite{Ro}, \cite{V1}, \cite{V2}, \cite{KaKu3}).

If $X$ is equipped with a holomorphic volume form $\omega$ (i.e. $\omega$ is a nowhere vanishing top holomorphic
differential form) then one can ask whether a similar approximation holds for automorphisms  and phase flows preserving $\omega$.
Under a mild additional assumption the answer is yes in the presence of the volume density property (VDP) which means that the Lie algebra
generated by complete homomorphic vector fields of $\omega$-divergence zero is dense in the space of all
holomorphic vector fields of $\omega$-divergence zero.

  The original method of Anders\'en and Lempert was developed further by Varolin and Toth \cite{TV1}, \cite{TV2}
who establish DP for a wide class of examples. Complex manifolds with VDP were harder to find
though Anders\'en proved VDP for $\C^n$ in \cite{A}  even before DP  for $\C^n$ was established in his paper with
Lempert \cite{AL}.

To deal with this difficulty we use technique of affine algebraic geometry and study the case when $X$ is a smooth affine algebraic variety  (over $\C$) 
while  whenever a volume form $\omega$ is present
it is an algebraic one. The following definitions are due to Varolin and the authors.

\begin{definition}\label{nc.08.20.10}
We say that $X$ has the algebraic  density property (ADP) if the Lie algebra $\LieA (X)$ generated by the set $\IVF (X)$ of complete
algebraic vector fields coincides with the space $\AVF (X)$ of all algebraic vector fields on $X$. Similarly in the presence of $\omega$
we can speak about the algebraic volume density property (AVDP) that means the equality $\LieAO (X)=\AVF_\omega (X)$ for analogous
objects (that is, all participating vector fields have $\omega$-divergence zero; say $\LieAO (X)$ is generated by $\IVF_\omega (X)$).
\end{definition}

It is worth mentioning that ADP and AVDP imply DP and VDP respectively (where the second implication is not that obvious) and
in particular all remarkable consequences for complex analysis on $X$.

An effective criterion whether $X$ has ADP was developed by the
authors in \cite{KaKu1}.  The main idea was to search for a nontrivial $\C [X]$-module inside $\LieA (X)$  using so-called pairs
of compatible vector fields. It lead in particular to the
proof of ADP for almost all homogeneous spaces of form $G/R$ where $G$ is any linear algebraic group and $R$ is a
closed reductive subgroup of $G$ \cite{DDK}
(one needs to have $R$ reductive for $G/R$ to be affine).

At first glance this idea does not work in the volume-preserving case. Indeed,
$\LieAO (X)$ cannot contain a $\C [X]$-module by the following reason. If $\nu \in \AVF_\omega (X)$ then for $f \in \C [X]$ the divergence
of the vector field $f \nu$ is computed by formula $\diver_\omega (f\nu ) = \nu (f)$, i.e. it is nonzero for a general $f$.

However  in this paper we establish a criterion  for the volume-preserving case whose effectiveness is
comparable with the one in \cite{KaKu1}  for ADP.
 Surprisingly  we do need to catch a  nontrivial  $\C [X]$-module but in a space different from $\LieAO (X)$.
To describe this space we need some extra notation.  Let $\cC_{k} (X)$ be the space of
algebraic differential $k$-forms on $X$ and $\cZ_{k}(X)$ and $\cB_{k} (X)$ be its subspaces of
closed and exact $k$-forms respectively. If $\dim X =n$ then there exists an isomorphism $\Theta : \AVF_\omega (X) \to \cZ_{n-1}(X)$
given by the formula $\xi \to  \iota_\xi \omega$ where $ \iota_\xi \omega$ is the interior product of $\omega$ and $\xi \in \AVF_\omega (X)$.
Consider the homomorphism $D_k : \cC_{k-1} (X) \to \cB_{k} (X)$ generated
by outer differentiation ${\rm d}$ and let $D=D_{n-1}$. The main theme of our new criterion is the search for a $\C [X]$-module in the space
$D^{-1} \circ \Theta ( \LieAO (X))$. With some additional assumptions the existence of such a module implies AVDP.  
%The situation is in fact even more pleasant because the methods we developed earlier in \cite{KaKu1} and \cite{DDK}
%for checking whether $\LieA (X)$ contains a $\C [X]$-module work perfectly for the space $D^{-1} \circ \Theta ( \LieAO (X))$.
  It is worth mentioning that the search of nontrivial $\C [X]$-modules in $D^{-1} \circ \Theta ( \LieAO (X))$ is based on the notion of semi-compatible 
pairs of vector fields
which is more transparent and accessible than the notion of compatible pairs used before.

The absence of such a simple approach was the main source of difficulties in \cite{KaKu4} where,  in particular, we proved AVDP for
all linear algebraic groups with respect to left (or right) invariant volume forms.  Now we are able not only to demonstrate a drastic simplification of the proof of this result
but to establish AVDP for every homogeneous space $G/R$ as before provided this space is equipped with a $G$-invariant volume form (actually,
we prove a stronger statement, see Theorem \ref{nc.11.23.20}, Remark \ref{nc.08.13.20}, and Corollary \ref{nc.12.15.20} below).

The paper is organized as follows. In section \ref{prel} we remind the definition and properties of semi-compatible vector fields which were first introduced in our previous work
\cite{KaKu1}. In section \ref{crit} we present the criterion that relates AVDP with such fields. In the next two sections we check whether some assumptions of this criterion are
valid for homogeneous spaces. Section 6 contains the main result about AVDP of a homogeneous space equipped with an invariant volume form.
As another application of our method in section \ref{surface}  we establish AVDP for an interesting class of surfaces that appeared in the solution to the  Gromov-Vaserstein problem \cite{IK}.
The appendix contains some rather straightforward material on invariant volume forms of homogeneous spaces.

 {\em Acknowledgements.} We are grateful to F. Forstneric and J. Globevnik for useful discussions and advice that lead to a better
presentation of our results.

\section{Preliminaries}\label{prel}

The main aim of this section is to remind the definition  of semi-compatible vector fields and
their properties . This notion
will be our main tool in the search of $\C[X]$-modules in the space $D^{-1} \circ \Theta ( \LieAO (X))$ discussed in the Introduction.

\begin{notation}\label{nc.07.13.10} %In what follows %(unless it is stated otherwise)
 For the rest of the paper $X$ is always a  {\bf smooth affine irreducible algebraic
variety over $\C$}   with an exception of Remark \ref{nc.01.13.10}  (where $X$ is normal) and all other notations mentioned in Introduction remain valid.
We consider often a situation when $X$ is equipped also with  an  {\bf effective}  algebraic action of a group $\Gamma$. We call such an $X$ a  $\Gamma$-variety
and for the rest of the paper $\Gamma$ is always  {\bf finite} .
The ring of regular $\Gamma$-invariant functions will be denoted by $\C [X, \Gamma ]$; it is naturally isomorphic to the ring
$\C [X/\Gamma ]$  of regular functions on the quotient space  $X/\Gamma$. \end{notation}

\begin{definition}\label{nc.08.02.30} 
(1) Recall that a holomorphic vector field $\xi$ on $X$ is called complete\footnote{In our previous papers
we called such fields completely (or globally) integrable.}
if there is a holomorphic $\C_+$-action $\Phi : \C \times X \to X$
such that $\xi (f)= {\frac{\rm d} {{\rm d} t}} f \circ \Phi (t, * )|_{t=0}$ for every $f \in \C [X]$.
This action $\Phi$ is called the phase flow of $\xi$. When $\Phi$ is an algebraic $\C_+$-action
the field $\xi$ is called locally nilpotent, and  when $\Phi$ factor through $\C^* \times X$ and generate
an algebraic $\C^*$-action then $\xi$ is called semi-simple.

(2)  Let $\xi$  and $\eta$ be
nontrivial complete algebraic vector fields on a    $\Gamma$-variety $X$ which are $\Gamma$-invariant.
We say that the pair $(\xi , \eta )$ is $\Gamma$-semi-compatible if\\
\begin{center} {\em the span of
$ (\Ker \xi \cap \C (X, \Gamma ))\cdot (\Ker \eta \cap \C (X, \Gamma )) $  contains a nonzero ideal of  $\C [X, \Gamma ]$.}\\[2ex]
\end{center}
 The largest ideal contained in the span will be called the associate $\Gamma$-ideal of the pair $(\xi , \eta )$.
In the case of a trivial $\Gamma$ we say that the pair $(\xi , \eta)$ is semi-compatible. In this terminology $\Gamma$-semi-compatibility
of $(\xi , \eta )$  is equivalent to  semi-compatibility of $(\xi' , \eta')$ where $\xi'$ and $\eta'$ are the vector fields on $X/\Gamma$
induced  by $\xi$ and $\eta$.

\end{definition}

\begin{example}\label{nc.11.16.00} (1)  Let $X_i$ ($i=1,2$) be an affine algebraic $\Gamma_i$-variety
and $p_i :  X:=X_1 \times X_2 \to X_i$ be the natural projection. Let $\Gamma_i$ and $\Gamma_1 \times \Gamma_2$ act naturally on $X$.
Suppose that   $ \xi_i$ is $\Gamma_i$-invariant complete algebraic
vector field on $X$ such that
$(p_j)_* ( \xi_i)=0$ for $1 \leq i \ne j \leq 2$. Since $(p_i)^*(\C [X_i, \Gamma_i ])\subset (p_i)^*(\C [X_i ]) \subset \Ker \xi_j$ we see that
$(\xi_1, \xi_2)$ is a $(\Gamma_1 \times \Gamma_2)$-semi-compatible pair on $X$ whose associate 
$\Gamma_1 \times \Gamma_2$-ideal is $\C [X,\Gamma_1 \times \Gamma_2 ]$.
 This argument does not work when one considers the question  of $\Gamma$-semi-compatibility of  $(\xi_1, \xi_2)$
for a finite subgroup $\Gamma$ of $\Gamma_1 \times \Gamma_2$. Nevertheless  if
the assumption of Proposition \ref{nc.08.13.10} below holds  
then $\xi_1$ and $\xi_2$ are $\Gamma$-semi-compatible.

(2) Consider $X=SL_2$ as a subvariety of $\C^4_{a_1,a_2,b_1,b_2}$  given by $a_1b_2-a_2b_1=1$ (where by $\C^n_{z_1, \ldots , z_n}$ we denote
a Euclidean space $\C^n$ with a fixed coordinate system $(z_1, \ldots ,z_n)$).
The vector fields
$$\xi = b_1\partial /\partial a_1   + b_2\partial /\partial a_2 \, \, \,{\rm and} \, \, \,  \eta = a_1\partial /\partial b_1 + a_2\partial /\partial b_2 $$
are locally nilpotent on $X$ and if we present every $A \in X$ as a matrix
\begin{center}
$A =  \left[
\begin{array}{rr}
a_1& a_2  \\
b_1 & b_2 \\
\end{array}  \right]$ \end{center}
then $\xi$ (resp. $\eta$ ) is associated with left multiplication by elements of the upper (resp. lower) triangular unipotent $\C_+$-subgroup $U$ (resp. $L$) of $SL_2$.
Note that $\C [b_1, b_2] \subset \Ker \xi$ and $\C [a_1, a_2] \subset \Ker \eta$ which implies that the pair $(\xi , \eta )$ is semi-compatible
 and the associate ideal is again $\C [X]$.
If $I$ is the identity matrix and $\Gamma$ is the subgroup $\{ I , -I \}$ of $SL_2$ then $\xi$ and $\eta$ induce
locally nilpotent vector fields $\xi'$ and $\eta'$ on $Y=X/\Gamma \simeq PSL_2$.   The fact  that the pair $(\xi', \eta')$ is semi-compatible is
less trivial but true by the  following.
\end{example}

\begin{proposition}\label{nc.08.13.10}\cite[Lemma 3.6]{KaKu1}
Let $X$ be a  $\Gamma$-variety, $\xi_1$ and $\xi_2$
be semi-compatible  $\Gamma$-invariant vector fields on $X$. Suppose that $\xi_1'$ and $\xi_2'$ are the induced vector fields on $X'=X/\Gamma$
and one of the following conditions holds.

{\rm (i)} $\xi_1$ and $\xi_2$ are locally nilpotent and they generate a Lie algebra isomorphic to
$\sgoth \lgoth_2$ that induces a non-degenerate algebraic $SL_2$-action on $X$ (i.e. general $SL_2$-orbits
are of dimension 3).

{\rm (ii)} $[\xi_1 , \xi_2 ]=0$ and $\xi_1$ is a locally nilpotent   vector field with
a finitely generated kernel   while $\xi_2$ is
either locally nilpotent (also with a finitely generated kernel) or semi-simple.

Then the pair $(\xi' , \eta' )$ is semi-compatible, i.e. $(\xi , \eta)$ is $\Gamma$-semi-compatible.

\end{proposition}

\begin{remark}\label{nc.02.20.13.2} The kernel of every semi-simple vector field is finitely generated,
i.e. the algebraic quotient of $X$ with respect to a $\C^*$-action is affine because $\C^*$ is reductive.
The kernel of a locally nilpotent vector field may not be finitely generated
by Nagata's counterexample  to the 14-th Hilbert problem. The assumption that $\Ker \xi_i$ is finitely generated
is missed in \cite[Lemma 3.4 and Lemma 3.6]{KaKu1} but it is essential for the proof.  However it does not
impact the results of \cite{KaKu1} since in case (i) the kernels of locally nilpotent derivations are automatically
finitely generated by the Hadziev theorem \cite{Had}. Similarly the assumption (ii) was used in situations when
locally nilpotent derivations associated with linear actions on Euclidean spaces, i.e. their kernels are again 
finitely generated  by the Maurer-Weitzenbšck's theorem \cite{Mau}, \cite{Wei}.
%The same is true for vector fields
%associated with linear $\C_+$-actions on Euclidean
%spaces \cite{Wei}.
The  locally nilpotent vector fields we are dealing  with  in this paper also have this property. Therefore we assume the following.
\end{remark}

\begin{convention}\label{nc.02.20.13.1}
For the rest of the paper every locally nilpotent vector field has a finitely generated kernel.

\end{convention}

The fact that locally nilpotent vector fields $\xi$ and $\eta$ generate a Lie algebra isomorphic to $\sgoth \lgoth_2$ does not imply
in general that they are semi-compatible. However we have the following.

\begin{proposition}[\cite{DDK}{[Theorem 12]}] \label{nc.11.16.10} Let $X$ admit a fixed point free non-degenerate $SL_2$-action
%(i.e. general $SL_2$-orbits are of dimension 3)
and let $U$ (resp. $L$) be the unipotent subgroup of upper (resp. lower) triangular matrices in $SL_2$ as in Example \ref{nc.11.16.00}  (2).
Suppose that $\xi$ and $\eta$ are the locally nilpotent
vector  fields   associated with the induced actions of the $\C_+$-groups $U$ and $L$ on $X$. Then $(\xi , \eta )$ is a semi-compatible pair.

\end{proposition}

\begin{remark}\label{nc.12.01.10} The assumption that the action is fixed point free and non-degenerate is essential
for the validity of Proposition \ref{nc.11.16.10}.
In the presence of fixed points  the statement does not hold (see \cite{Do})  and for the degenerate case
$SL_2/\C^*$ provides a counterexample because of the following.

\end{remark}

\begin{proposition}\label{nc.12.15.10} Let $X$ be a smooth affine surface different from $\C^2$, $\C^* \times \C^*$, or $\C^* \times \C$. Then
$X$ does not admit a semi-compatible pair of algebraic vector fields that are locally nilpotent or semi-simple.

\end{proposition}

\begin{proof}
Assume that $\xi_1$ and $\xi_2$ are such vector fields on $X$. Let $H_i$ be the $\C_+$ or $\C^*$
group associated with $\xi_i$ acting on $X$ and let $\rho_i : X \to X_i :=X//H_i$ be the quotient morphism.
If $X_1$ is a point or if it is a curve different from $\C$ and $\C^*$ then $X_1$ does not admit nonconstant homomorphisms from $\C$ or $\C^*$
and therefore any general orbit of $H_2$ must be contained in a fiber of $\rho_1$. Hence morphism
$\rho =(\rho_1, \rho_2): X \to X_1 \times X_2$ is not birational contrary to the fact that semi-compatibility is
equivalent to the claim that $Z=\rho (X)$ is closed in $X_1\times X_2$ and $\rho$ is finite birational
by Proposition 3.4 in \cite{KaKu1}. Thus $Z=X_1\times X_2$ is one of the surfaces listed in the statement of Proposition
and the Zariski Main theorem implies that $X$ is isomorphic $Z$.
\end{proof}

The existence of fixed point free non-degenerate $SL_2$ actions on homogeneous spaces is guaranteed by the next result.

\begin{proposition}\cite[Theorem 24]{DDK}\label{nc.11.16.20} Let $G$ be a semi-simple group and $R$ be a proper closed reductive subgroup
of $G$ such that $X=G/R$ is at least three-dimensional. Then there exists an $SL_2$ subgroup of $G$
whose natural action on $X$  is fixed point free and non-degenerate.
In particular, $X$ admits semi-compatible pairs of locally nilpotent vector fields.

\end{proposition}

Besides the properties of  semi-compatible pair listed above we need also the following. 
%consequence of the
%Nakayama lemma.

\begin{lemma}\label{nc.12.27.10} Let $A$ be an affine domain (over $\C$) and $M$ be a finitely generated $A$-module.
%equipped
%with a filtration $M=\bigcup_{i \in \Nat } M_i, \, M_i  \subseteq  M_{i+1}$ such that each $M_i$ is a finitely generated module
%over $A$.
Let $N$ be a submodule of $M$.
 Suppose that for every  maximal ideal $\mu$ of $A$ one has
%and each $i\in \Nat$ one has

%{\rm (i)} $\mu M \cap M_i = \mu M_i$;

%{\rm (ii)}
$$M/\mu M = N /( \mu M \cap N).$$

Then $M=N$. \end{lemma}

\begin{proof}
%It suffices to show that $M_i =N_i$ for every $i\in \Nat$ where $N_i=N\cap M_i$, i.e. we can suppose that $M$ itself is finitely generated.
The Nullstellensatz implies that $M=N$ iff the localized versions of this equality with respect to every $\mu$ are valid.
Thus it is enough to consider the case when $A$ is a local ring with the only maximal ideal $\mu$.
%Furthermore, for $N_i=N\cap M_i$ assumptions (i) and (ii) imply that $M_i/\mu M_i = N_i / (\mu M_i \cap N_i)$ for every $i \in \Nat$, i.e. it suffices to
%prove the statement in the case of a finitely generated $M$.
Then for  the $A$-module $L= M/N$ the equality  $M/\mu M = N / (\mu M \cap N)$ implies that $L = \mu L$, i.e. by the Nakayama lemma $L$ is the zero module
which is the desired conclusion.
\end{proof}

We are going to apply Lemma  \ref{nc.12.27.10} is the case of $X$ equipped with two semi-simple vector fields $\xi$ and $\eta$
such that they commute. That is, this pair of fields generate a $\C^*\times \C^*$-action on $X$ and
a structure of a $\Z^2$-graded algebra $A =\oplus_{\alpha \in \Z^2} A_\alpha$ on $A=\C [X]$ where $\Z^2$ is the group of
characters for $\C^*\times \C^*$.  Note that for $\alpha = (m,n)$
a function $f$ belongs to $A_\alpha$ iff $\xi (f)=mf$ and $\eta (f)=nf$,  i.e., the grading is preserved by the vector fields
and $\Ker \xi =\oplus_{n \in \Z} A_{(0,n)}$ and $\Ker \eta =\oplus_{m \in \Z} A_{(m,0)}$. Furthermore, $g \in \Ker \xi$ iff $g_\alpha \in \Ker \xi$
for every $\alpha \in \Z^2$ where $g= \sum_{\alpha \in \Z^2} g_\alpha$ and $g_\alpha \in A_\alpha$ (similarly for $\eta$).

\begin{lemma}\label{nc.04.03.13.2}  Let $X, \xi, \eta$, and  $A =\oplus_{\alpha \in \Z^2} A_\alpha$ be  as above 
and let  $\xi$ and $\eta$ be tangent to a closed subvariety $Y$ of $X$. That is,
one has a grading  $B=\oplus_{\alpha \in \Z^2} B_\alpha$ on $B= \C [Y]$. Then for every $b \in B_\alpha$
there exists $a\in A_\alpha$ such that $a|_Y=b$. In particular for every $b \in \Ker \xi |_Y$ (resp. $b \in \Ker \eta |_Y$)
there is $a\in \Ker \xi$ (resp. $a  \in \Ker \eta$) such that $a|_Y =b$. 
\end{lemma}

\begin{proof}
Every regular function $b \in B_\alpha$ has a regular extension $f$ to the ambient algebraic variety $X$. Taking
the $\alpha$-homogeneous component of $f$ we get the desired $a$. 
\end{proof}

\begin{proposition}\label{nc.04.05.13.3} Let $\xi$ and $\eta$ be two semi-simple vector fields on $X$
such that they commute, i.e. they induced a $\C^*\times \C^*$-action on $X$. Suppose that
the algebraic quotient morphism $\pi : X \to W :=X//(\C^* \times \C^*)$ is smooth and all orbits of the action are two-dimensional.
Then the pair $(\xi , \eta )$ is semi-compatible.
\end{proposition}

\begin{proof} Since $\C^*\times \C^*$
is reductive this implies that for any $w \in W$ the fiber $\pi^{-1}(w)$ contains a unique closed orbit $Y$
and $\pi^{-1} (w) \setminus Y$ is the union of nonclosed orbits each of which contains $Y$ in its closure (e.g., see \cite{Kraft} or \cite{MuFo}).
By assumption each orbit is two-dimensional and it cannot contain another orbit in its closure, i.e. $ \pi^{-1}(w)=Y \simeq\C^* \times \C^*$.
Consider $A =\oplus_{\alpha \in \Z^2} A_\alpha$
and $B=\oplus_{\alpha \in \Z^2} B_\alpha$ as in Lemma \ref{nc.04.03.13.2}. Note that $A_{(0,0)}=\C [W]$,
i.e. $A$ and every $A_\alpha$ are $\C [W]$-modules. Furthermore, since $\pi$ is smooth, $A/\mu (w)$ 
is a domain where $\mu (w)$
is the maximal ideal in $\C [W]$ associated with the point $w$. That is, $B=A/\mu (w)$  and by Lemma \ref{nc.04.03.13.2}
$B_\alpha=A_\alpha/\mu (w)$.  Let $A'$ (resp. $A'_\alpha$) be the $\C [W]$-submodule of $A$ (resp. $A_\alpha$)
generated by ${\rm Span} \Ker \xi \cdot \Ker \eta$ (resp. $\sum_{\beta + \gamma =\alpha} {\rm Span} (\Ker \xi \cap A_\beta )\cdot (\Ker \eta \cap A_\gamma )$).
Note that $A'=\oplus_{\alpha \in \Z^2} A_\alpha'$ and consider $B_\alpha' = A_\alpha' /\mu (w)$.
The restrictions $\xi_w$ and $\eta_w$ of $\xi$ and $\eta$ to $Y$ are semi-compatible, i.e.  ${\rm Span} \Ker \xi_w \cdot \Ker \eta_w =B=\oplus_{\alpha \in \Z^2} B_\alpha$ 
and the second statement of Lemma \ref{nc.04.03.13.2} implies that  $B_\alpha'=\sum_{\beta + \gamma =\alpha}
{\rm Span} (\Ker \xi_w \cap B_\beta )\cdot (\Ker \eta_w \cap B_\gamma )$.
That is, $A_\alpha/\mu (w)=B_\alpha=B_\alpha' = A_\alpha' /\mu (w)$ for every $w \in W$. Since each vector space $B_\alpha$ is at most one-dimensional,
$A_\alpha$ is a finitely generated $\C [W]$-module. By Propositions \ref{nc.12.27.10} $A_\alpha =A_\alpha'$ and thus $A=A'$
which is the desired conclusion. 
\end{proof}

\section{The criterion}\label{crit}

To demonstrate the strength of our method we start this section with a simplified version of our criterion which yields AVDP   for semi-simple Lie groups in Example
\ref{ss}
(this fact was established in \cite{KaKu4} but our present approach is strikingly simpler than the argument used in that paper).

Recall that
$$\Theta : \AVF_\omega (X) \to \cZ_{n-1}(X)$$ is the isomorphism
given by the formula $\xi \to  \iota_\xi \omega$ (where $ \iota_\xi \omega$ is the interior product of $\omega$ and $\xi \in \AVF_\omega (X)$)
and  $D_k : \cC_{k-1} (X) \to \cB_{k} (X)$ is the homomorphism generated
by outer differentiation ${\rm d}$. The next simple observation provides a crucial connection between semi-compatibility and existence of $\C [X]$-modules
in $D^{-1} \circ \Theta (\LieAO (X))$ where $D=D_{n-1}$.

\begin{proposition}\label{bracket} Let $\xi$ and $\eta$ be vector fields from   $\AVF_\omega (X )$ . \footnote{Actually the statement
remains valid when $X$ is a complex manifold equipped with a holomorphic volume
form $\omega$ and $\xi$ and $\eta$ are two holomorphic vector fields on $X$ of $\omega$-divergence zero.}
Then

$$\iota_{[\xi, \eta]} \omega   =  {\rm d} \iota_{\xi} \iota_{\eta} \omega . \leqno{(1)}$$

\end{proposition}

\begin{proof}
Recall the following relations between the outer differentiation $\rm d$, Lie derivative $L_\xi$ and
the interior product $\iota_\xi$ \cite{KoNo}
$$ L_\xi = {\rm d}  \iota_\xi + \iota_\xi  {\rm d} \, \, \, {\rm and} \, \, \, [L_\xi , \iota_\eta]
=\iota_{[\xi , \eta ]}. \leqno{(2)}$$
By this formula
$$\iota_\xi {\rm d} \iota_\eta  \omega = \iota_\xi (L_\eta - \iota_\eta {\rm d})\omega $$
where the right-hand side is zero since $L_\eta \omega  - \iota_\eta {\rm d}\omega =0$
for closed $\omega$ and $\eta$ of $\omega$-divergence zero. Then another application of
formula (2) in combination with the fact that $\iota_\xi {\rm d} \iota_\eta \omega =0$ yields
$$[L_\xi , \iota_\eta]\omega = L_\xi \iota_\eta \omega  - \iota_\eta L_\xi \omega = L_\xi \iota_\eta \omega =
{\rm d} \iota_{\xi} \iota_{\eta} \omega + \iota_\xi {\rm d} \iota_\eta \omega  =
{\rm d} \iota_{\xi} \iota_{\eta} \omega .$$ Thus by formula (2) we have the desired
equality $$\iota_{[\xi, \eta]} \omega   =  {\rm d} \iota_{\xi} \iota_{\eta} \omega .$$

\end{proof}

Let  $\xi, \eta \in \IVF_\omega (X )$,  $f \in \Ker \xi$, and $g \in \Ker \eta$.  Replacing
$\xi$ and $\eta$ in Formula (1) by complete fields  $f \xi$ and $g \eta$ respectively we can see that $(fg)  \iota_{\xi}\iota_{\eta} \omega \in D^{-1} \circ \Theta (\LieAO (X))$.
Hence one has the following.

\begin{corollary}\label{nc.07.13.20basic} Let  $X$ be a  variety equipped with an algebraic volume form $\omega$ and let
$\xi$ and $\eta$ be semi-compatible
divergence-free vector fields on $X$.
Then $D^{-1} \circ \Theta (\LieAO (X ))$ contains
a nontrivial $\C [X ]$-submodule $L$ of the module  $\cC_{n-2} (X)$. \footnote{
In a more general setting
an accurate proof of Corollary \ref{nc.07.13.20basic}
is given below in Corollary \ref{nc.07.13.20} (similarly we only sketch the proofs of Proposition \ref{nc.07.13.30basic} and Theorem \ref{nc.07.14.10basic}
since they are special cases of Proposition \ref{nc.07.13.30} and Theorem \ref{nc.07.14.10} respectively).}

\end{corollary}

Let $\mu (x)\subset \C [X]$ be the maximal ideal of functions vanishing at  $x \in X$ and let $L$ be
the largest $\C [X ]$-submodule of  $D^{-1} \circ \Theta (\LieAO (X ))$. Since $\cC_{n-2} (X)$ is a finitely generated
$\C [X]$-module, Lemma \ref{nc.12.27.10} implies that  equality $L=\cC_{n-2} (X)$ holds as soon
as  $L/\mu (x) L =\cC_{n-2} (X)/ \mu (x) \cC_{n-2} (X)$ for every $x \in X$. The latter is true provided  Condition (A$'$) below holds 
and we have the following.

\begin{proposition}\label{nc.07.13.30basic} Let  $X$ be a  variety equipped with an algebraic volume form $\omega$
and let $(\xi_j , \eta_j )_{j=1}^k$ be pairs of divergence-free
semi-compatible vector fields.
Let $I_j$ be the ideal associated with $(\xi_j , \eta_j )$,
and let $I_j(x)= \{ f (x) | f \in I_j \}$ for $x \in X$.
Suppose that \\

\noindent {\rm (A$'$)} \hspace{0.3cm} for every $x \in X$ the set $\{ I_j (x) \xi_j (x) \wedge \eta_j (x) \}_{j=1}^k$
generates the fiber $\Lambda^2 T_{x}X$ of $\Lambda^2 TX$ over $x$. \\

Then $\Theta (\LieAO (X))$ contains $\cB_{n-1} (X)$ .
\end{proposition}

As a consequence of Proposition \ref{nc.07.13.30basic} we have (the basic form of) our main criterion.

\begin{theorem}\label{nc.07.14.10basic}  Let  $X$ be a  variety equipped with an algebraic volume form $\omega$
and pairs of divergence-free
semi-compatible vector fields satisfying Condition (A$'$) from Proposition \ref{nc.07.13.30basic}.
Suppose also that the following condition is true\\

\noindent {\rm (B$'$)} \hspace{0.3cm} the image of  $\Theta (\LieAO (X))$ under
De Rham homomorphism $\Phi_{n-1} : \cZ_{n-1} (X) \to H^{n-1} (X, \C)$ coincides with  %the subspace  $\Phi_{n-1}( \cZ_{n-1} (X))$  of
 $H^{n-1}(X , \C)$.\\

Then
$\Theta (\LieAO (X)) = \cZ_{n-1} (X)$  and therefore
$\LieAO (X) = \AVF_\omega (X)$, i.e., $X$ has the algebraic volume density property.

\end{theorem}

\begin{example}\label{ss}
Let $G$ be a  semi-simple  group  of dimension $n$. % has AVDP with respect to the Haar form:
For a simply connected $G$
the cohomology ring $H^*(G, \C )$ is isomorphic to $H^*(S^{k_1} \times \ldots \times S^{k_l}, \C )$ where $S^k$ is
a sphere of dimension $k$ and each $k_i \geq 3$ (e.g., see Theorem from Section 2.1 in \cite{Fuchs}).
Therefore (even in a non-simply connected case) $H^{n-1}(G, \C ) =0$ and  condition (B$'$) holds automatically.
The semi-compatible pairs of divergence-free vector fields on $SL_2$ from Example  \ref{nc.11.16.00}(2)  taken over all $SL_2$-subgroups of $G$ fulfill condition (A$'$)
(see also Lemma \ref{nc.08.10.10}  below and the discussion thereafter). Thus $G$ has AVDP  by Theorem \ref{nc.07.14.10basic}.

\end{example}

To extend this type of argument to homogeneous spaces one needs to consider, for instance, the case of $X=G/\Gamma$ where $\Gamma$
is a finite subgroup of $G$. Note that any vector field on $X$ can be viewed as a $\Gamma$-invariant vector field on $G$. This leads to
a theory of algebraic volume density for $\Gamma$-invariant fields and a more general setting for our criterion.

%\begin{notation}
\begin{sit}\label{nc.11.16.30}{\bf Notation.} {\rm Let $X$ be a
$\Gamma$-variety equipped with an algebraic volume form $\omega$.
Then for every $\gamma \in \Gamma$ one has
$$\omega \circ \gamma = \chi_\omega (\gamma ) \omega$$ where
$\chi_\omega$  is a map from  $\Gamma$  into the group of invertible regular functions on $X$.
We denote by $\LieAO (X , \Gamma)$
the Lie algebra generated by $\Gamma$-invariant complete algebraic vector fields of $\omega$-divergence zero.
Similarly $\AVF_\omega (X, \Gamma )$ will be the space of algebraic $\Gamma$-invariant divergence-free vector fields.
Such fields can be also viewed as elements of $\AVF_\omega (X/\Gamma )$. By  $\cC_k (X, \Gamma )$  we denote the space
%of $\Gamma$-quasi-invariant
of algebraic $k$-forms  $\alpha$ such that
$$\alpha \circ \gamma = \chi_\omega (\gamma ) \alpha $$ for every $\gamma \in \Gamma$. %corresponding to the character  $\chi_\omega$ .
By   $\cZ_k (X, \Gamma )$
and $\cB_k (X, \Gamma )$ we denote
the subspaces of closed and exact forms in $\cC_k (X, \Gamma )$.
The restriction of $\Theta : \AVF_\omega (X) \to \cZ_{n-1}(X)$ generates
an isomorphism between $\AVF_\omega (X, \Gamma )$ and  $\cZ_{n-1} (X, \Gamma )$.
 We assume further that $\chi_\omega$ is a character because we shall deal with the situation when
$\Gamma$ is contained in a connected group of automorphisms and the following holds.}
\end{sit}

\begin{proposition}\label{nc.02.24.13.3} Let $F$  be a connected group of algebraic automorphisms
of a smooth irreducible affine algebraic  variety $Y$ equipped with a volume form $\tau$. Then for every $\alpha \in F$
one has $\alpha^* (\tau) =c \tau$ for some $c \in \C^*$.

\end{proposition}

\begin{proof} Note that  $\alpha^* (\tau) =h \tau$ where $h$ is an element of the group $H$ of invertible functions on $Y$.
Note that $h$ must be in the same connected component of $H$ as constants. However
the group $H/\C^*$ is isomorphic to $H^1 (Y, \Z$) (e.g., see \cite{Fu}) and therefore it is discrete. Hence $h\in \C^*$.

\end{proof}

\begin{remark}\label{nc.02.13.13} In general if $\alpha$ does not belong to the connected component of identity in
the algebraic automorphism group the assumption that $\alpha^* (\tau) / \tau$ is constant
does not hold. Consider for instance a torus $(\C^*)^2$ equipped with coordinates $(z,w)$. Let
$\tau = ( \frac{z}{w} )\frac{{\rm d} z}{z} \wedge \frac{{\rm d} w} {w}$ and
$\alpha$ be given by $(z,w) \to (w,z)$. Then $\alpha^* (\tau) = -\frac{w^2}{z^2} \tau$.
However if $Y$ is weakly rationally connected in terminology of \cite[Proposition 6.4]{KaKu4}
then $\alpha^* (\tau) / \tau$ is constant for every algebraic automorphism $\alpha$.
\end{remark}

Since $\chi_\omega$ is a character $D_k$ sends  $\cC_{k-1} (X, \Gamma )$ into  $\cB_{k} (X, \Gamma )$ and even more.

\begin{lemma}\label{nc.08.13.40} Let Notation \ref{nc.11.16.30} hold and $\chi_\omega$  be a character, i.e.
$\omega \circ \gamma /\omega$ is constant for every $\gamma \in \Gamma$. Then
one has  $\cB_{k} (X, \Gamma ) =D_k (\cC_{k-1} (X, \Gamma ))$  for every $k\geq 1$.
In particular, if $\cC_{n-2} (X, \Gamma ) = D^{-1} \circ \Theta (\LieAO (X , \Gamma))$ then $\cB_{n-1} (X, \Gamma )\subset  \Theta (\LieAO (X , \Gamma)$.
\end{lemma}

\begin{proof} Let $\Xi$ be the  set of irreducible characters of representations of $\Gamma$. Then  $\cB_{k} (X)  =\oplus_{\chi \in \Xi} V_{\chi}$
and   $\cC_{k-1} (X)  =\oplus_{\chi \in \Xi} U_{\chi}$ where $V_{\chi}$ and $U_{\chi}$ are the subspaces corresponding
to character $\chi$.  Since $D_k(U_{\chi}) \subset V_{\chi}$ and $\cB_{k} (X) =D_k (\cC_{k-1} (X))$
one has $D_k(U_{\chi}) = V_{\chi}$  which implies the desired conclusion.

\end{proof}

Now Proposition \ref{bracket} implies the following.

\begin{corollary}\label{nc.07.13.20} Let  $X$ be a    $\Gamma$-variety equipped with an algebraic volume form $\omega$ and let
$\xi$ and $\eta$ be $\Gamma$-semi-compatible
 divergence-free vector fields on $X$ with an associated $\Gamma$-ideal $I \subset \C [X, \Gamma ]$.
Then $D^{-1} \circ \Theta (\LieAO (X , \Gamma))$ contains
the $\C [X, \Gamma ]$-submodule $I  \iota_{\xi}\iota_{\eta} \omega$  of the module $\cC_{n-2} (X, \Gamma )$.
\end{corollary}

\begin{proof}
Formula (1) from Proposition \ref{bracket}
implies that for $f \in \Ker \xi$, $g \in \Ker \eta$, and $\kappa =[f\xi , g \eta ]$
one has $\iota_{\kappa} \omega = {\rm d}  \iota_{f\xi}\iota_{g\eta} \omega$, i.e.
$(fg)  \iota_{\xi}\iota_{\eta} \omega \in D^{-1} (\iota_{\kappa} \omega) $.
Choose $\Gamma$-invariant $f$ and $g$.
Then by $\Gamma$-semi-compatibility
$D^{-1} \circ \Theta (\LieA^\omega (X))$ contains the $\C [X, \Gamma]$-module $I \iota_{\xi}\iota_{\eta} \omega$.
\end{proof}

\begin{notation}\label{nc.12.28.10}
Suppose that  $X$ is  an  $n$-dimensional $\Gamma$-variety equipped with a volume  form $\omega$.
For any subbundle $V$ of $\Lambda^k TX$
we denote by   $\cC_{n-k}^V (X, \Gamma )$ the ``dual" subbundle   of $\cC_{n-k}(X, \Gamma )$   that consists of all forms $\tau$ such that
for every $x \in X$ the restriction
of $\tau$ to $\Lambda^{n-k} T_xX$ is contained in the space generated by  elements  $\iota_{v_1} \ldots \iota_{v_k} \omega (x)$
where $v_1 \wedge \ldots \wedge v_k$ runs over  $V_x$.
\end{notation}

\begin{proposition}\label{nc.07.13.30} Let  $X$ be a $\Gamma$-variety equipped with an algebraic volume form $\omega$
and let $(\xi_j , \eta_j )_{j=1}^k$ be pairs of
$\Gamma$-semi-compatible  divergence-free  vector fields.
Let $I_j$ be the $\Gamma$-ideal associated with $(\xi_j , \eta_j )$,
and let $I_j(x)= \{ f (x) | f \in I_j \}$ for $x \in X$.
Suppose that $V$ is a subbundle of $\Lambda^2 TX$ such that \\

\noindent  {\rm (A'')}  \hspace{0.3cm} for every $x \in X$ the set $\{ I_j (x) \xi_j (x) \wedge \eta_j (x) \}_{j=1}^k$
generates the fiber $V_{x}$ of $V$ over $x$. \\

Then $\Theta (\LieAO (X, \Gamma ))$ contains $D (\cC_{n-2}^V (X, \Gamma))$.
In particular, if  one has additionally 
$$V=\Lambda^2 TX $$   (in which case we say that   Condition (A) holds)
then $\Theta (\LieAO (X, \Gamma ))$ contains $\cB_{n-1} (X, \Gamma)$ .
\end{proposition}

\begin{proof} By Corollary \ref{nc.07.13.20} $D^{-1} \circ \Theta (\LieAO (X, \Gamma ))$ contains
a $\C [X, \Gamma ]$-module $L$ of form $$L= \sum_{j=1}^k I_j \iota_{\xi_j}\iota_{\eta_j} \omega .$$
Treat now $\cC_{n-2}^V (X, \Gamma )$ as a  finitely generated $\C [X/\Gamma ]$-module, and $L$ as a submodule of this module .
 By Lemma \ref{nc.12.27.10}  the module $\cC_{n-2}^V (X, \Gamma )$ coincides with $L$ if for every $x \in X/\Gamma $ one
has $$\cC_{n-2}^V (X, \Gamma)/ (\mu (x)\cC_{n-2}^V (X, \Gamma)) = L / (\mu (x)L)$$
where $\mu (x)$ is the maximal ideal in $\C [X/\Gamma ]$ associated
with the point $x$.
The last equality is, of course, equivalent to Condition (A'').
Thus we have the desired conclusion.

\end{proof}

\begin{corollary}\label{nc.01.01.10}  Let  $X$ be a $\Gamma$-variety equipped with an algebraic volume form $\omega$
and let $(\xi_j , \eta_j )_{j=1}^{k_0}$ be pairs of divergence-free
$\Gamma$-semi-compatible vector fields.
Suppose that  $F$
is a group of algebraic automorphisms of $X$ with the following properties: 

{\rm (i)} the natural $F$-action commutes with the $\Gamma$-action and preserves $\omega$ up to constant factors;

{\rm (ii)} $F$ induces a transitive action on $X/\Gamma$ \footnote{Condition (ii) implies that the $\Gamma$-action is free.}.

Suppose that $V$ is the subbundle of $\Lambda^2 TX$ such that it is invariant under the natural $F$-action
and at a general point $x_0 \in X$ its fiber $V_{x_0}$ is generated by the set
$\{ \xi_j(x_0) \wedge \eta_j(x_0) \}_{j=1}^k$.
Then $\Theta (\LieAO (X, \Gamma ))$ contains $D (\cC_{n-2}^V (X, \Gamma))$  and therefore
by Lemma \ref{nc.08.13.40} $\Theta (\LieAO (X, \Gamma ))\supset \cB_{n-1} (X, \Gamma)$
provided $\chi_\omega$ is a character. 
\end{corollary}

\begin{proof}
Since $x_0$ is a general point $I_j(x_0)=\{ f(x_0) | f \in I_j \} \ne 0$ for every $j$
where $I_j$ is the $\Gamma$-ideal associate with $(\xi_j , \eta_j)$. In particular,
$\{ I_j (x_0) \xi_j (x_0) \wedge \eta_j (x_0) \}_{j=1}^{k_0}$
generates $V_{x_0}$.
Since every automorphism $\alpha \in F$ sends $\omega$ into $c \omega , \, c\in \C^*$,
it  preserves the set of vector fields with $\omega$-divergence zero. Complete  algebraic vector fields
are  also preserved by $\alpha$. In combination with
transitivity of the $F$-action on $X/\Gamma$ this implies that we can extend
the sequence to $(\xi_j , \eta_j )_{j=1}^{k_0}$ to  a larger sequence of pairs of semi-compatible
divergence-free vector fields  of  form $\{ (\alpha_*(\xi_j) , \alpha_*(\eta_j) ) | j=1, \ldots , k_0; \,\alpha \in F \} $  for which Condition (A) from Proposition \ref{nc.07.13.30}
holds. Hence we are done.

\end{proof}

As a consequence of Proposition \ref{nc.07.13.30} we have our main criterion.

\begin{theorem}\label{nc.07.14.10}  
Let  $X$ be a $\Gamma$-variety equipped with an algebraic volume form $\omega$
and let $(\xi_j , \eta_j )_{j=1}^k$ be pairs of
$\Gamma$-semi-compatible  divergence-free  vector fields.
Let $I_j$ be the $\Gamma$-ideal associated with $(\xi_j , \eta_j )$
and $I_j(x)= \{ f (x) | f \in I_j \}$ for $x \in X$. Suppose that\\

\noindent {\rm (A)}  \hspace{0.3cm} for every $x \in X$ the set $\{ I_j (x) \xi_j (x) \wedge \eta_j (x) \}_{j=1}^k$ generates  $\Lambda^2 T_{x}X$.  \\

Suppose also that the following condition is true\\
 
\noindent {\rm (B)} \hspace{0.3cm} the image of  $\Theta (\LieAO (X,\Gamma))$  under
De Rham homomorphism $\Phi_{n-1} : \cZ_{n-1} (X) \to H^{n-1} (X, \C)$ coincides with the subspace  $\Phi_{n-1}( \cZ_{n-1} (X, \Gamma))$  of $H^{n-1}(X , \C)$.\\

Then
$\Theta (\LieAO (X, \Gamma )) = \cZ_{n-1} (X, \Gamma)$  and therefore
$\LieAO (X, \Gamma ) = \AVF_\omega (X, \Gamma)$.

\end{theorem}

\begin{definition}\label{nc.11.20.10} (1) If $X$ is a    $\Gamma$-variety equipped with an algebraic volume form $\omega$
and the equality $\LieAO (X, \Gamma ) = \AVF_\omega (X, \Gamma)$ holds then we say that $X$ has $\Gamma$-AVDP
(with respect to $\omega$).

(2) If under the assumption of Theorem \ref{nc.07.14.10} the group $\Gamma$ is trivial (or if it acts trivially)
then we say that the pair $(X, \omega )$ satisfies Condition (B). In the case of a nontrivial $\Gamma$-action
we speak about the validity of Condition (B) for the triple $(X, \omega , \Gamma )$.

(3) Note that $\omega$ induces on  $X_0=X/\Gamma$ a multi-valued differential form $\omega_0$ such that
it does not vanish and locally any two branches of $\omega_0$ differ by a constant factor. We call such
$\omega_0$ a multi-volume form on $X_0$.  Since divergence-free algebraic vector fields are the same
for proportional volume forms one can define such the space $\AVF_{\omega_0} (X_0)$ of
such fields with respect to $\omega_0$ (they are nothing but
the fields on $X_0$ induced by $\AVF_\omega (X, \Gamma)$).  Similarly one can define
$\IVF_{\omega_0}$ and $\Lie_{\rm alg}^{\omega_0} (X_0 )$.  In this sense we can speak 
about AVDP of $X_0$ with respect to $\omega_0$ which means the equality
$\Lie_{\rm alg}^{\omega_0} (X_0 ) = \AVF_{\omega_0} (X_0)$.
If the $\Gamma$-action on $X$ extends to a $\Gamma'$-action for a finite group $\Gamma'$ 
such that $\Gamma$ is a normal subgroup of $\Gamma'$ then $X_0$ is equipped with an action of $\Gamma_0=\Gamma'/\Gamma$, 
and we can speak in a similar manner about $\Gamma_0$-AVDP of $X_0$ which is the same as AVDP for $X_0/\Gamma_0$ with respect
to the multi-volume form induced on $X_0/\Gamma_0$.

\end{definition}

Sections 4-6  of the paper are devoted to the application of our criterion to prove AVDP for homogeneous spaces.

\section{Condition (B) for homogeneous spaces}

Actually,  in this section  we are checking at first glance a stronger condition that the image of $\Theta (\IVF_\omega (X))$ under
De Rham homomorphism $\Phi_{n-1} : \cZ_{n-1} (X) \to H^{n-1} (X, \C)$ coincides with $H^{n-1} (X, \C )$. However
it is equivalent to Condition (B) by the following proposition.

\begin{proposition}\label{nc.08.11.10} Let $X$ and $\omega$ be as in  Notation  \ref{nc.11.16.30} and
$N$ be the subspace of $\LieAO (X)$ generated by elements that can be presented as Lie brackets  of other elements. Then

{\rm (1)} $\Theta (N)$ is contained in the kernel of $\Phi_{n-1}$,

 {\rm (2)} the image of $\Theta (\IVF_\omega (X))$ under $\Phi_{n-1}$ generates $H^{n-1} (X, \C)$
 provided pair $(X, \omega )$ satisfies Condition (B).

 \end{proposition}

%
 %\begin{remark} \label{cohom}  The previous Lemma shows the following: If $X$ has AVDP then  the image of $\Theta (\IVF_\omega (X))$ under $\Phi_{n-1}$ generates $H^{n-1} (X, \C)$.
 %\end{remark}
 %

 \begin{proof} The second statement is a consequence of the first one which in turn follows from
Formula (1).
%the formula for the Lie bracket of two vector fields $\nu$ and $\mu$ on $X$
%$$[\nu , \mu ] =\lim_{t \to  0} \frac {\varphi^*_t (\nu ) - \nu }{t}$$ where $\varphi_t$
%is the phase flow generated by $\mu$. Indeed, let $\xi_t =\varphi^*_t (\nu ) - \nu$ and $\nu$ and
%$\mu$ be in $\LieAO (X)$.  Suppose that $\alpha_t = \Theta (\xi_t)/t$ for $ t \ne 0$, and $\alpha_0 =\Theta ([\nu , \mu ])$.
%Then $t\alpha_t = \Theta (\xi_t) =\Theta (\varphi^*_t (\nu )) - \Theta ( \nu )$ is contained in the kernel
%of $\Phi_{n-1}$ since the phase flow $\varphi_t$ yields a trivial action on $H^*(X)$. Thus $\alpha_t$ belongs
%to the this kernel for $t \ne 0$. The subspace of exact forms $\cB_{n-1} (X)$ is closed (with respect
%to the compact-open topology)
%in the space $\cZ_{n-1} (X)$ because by the De Rham theorem a closed form is exact iff its integrals
%over a finite number of singular cycles vanish. Hence $\alpha_0 =\lim_{t\to 0} \alpha_t$ is also in the kernel of $\Phi_{n-1}$
%which implies the first statement.

 \end{proof}

\begin{lemma}\label{nc.07.15.15} Let $T$ be a torus $(\C^*)^n$. Then Condition (B)  is
valid for an invariant volume form on $T$.

\end{lemma}

\begin{proof} Equip $T=(\C^*)^n$ with a natural coordinate system
$(z_1, \ldots , z_n)$, i.e. an
invariant volume form can be chosen as $\omega = \frac {{\rm d} z_1}{z_1}\wedge \ldots \wedge  \frac {{\rm d} z_n}{z_n}$.
By the K\"unneth formula $H^{n-1} (T, \C ) \simeq \Z^n$.
Consider the semi-simple vector fields $\nu_i = z_i \frac {\partial} {\partial z_i}$, i.e. $\nu_i$ is associated
with the $\C^*$-action $(z_1, \ldots , z_i , \ldots , z_n) \to (z_1, \ldots , \lambda z_i , \ldots , z_n)$ preserving
$\omega$. Hence $\nu_i$ has the $\omega$-divergence zero. It remains to note that
$$\iota_{\nu_i} \omega =  \frac {{\rm d} z_1}{z_1}\wedge \ldots \wedge  \frac {{\rm d} z_{i-1}}{z_{i-1}}
 \wedge \frac {{\rm d} z_{i+1}}{z_{i+1}}\wedge \ldots \wedge  \frac {{\rm d} z_n}{z_n}$$ for $i=1, \ldots , n$
and that such an $n$-tuple of $(n-1)$-forms yield a basis in $H^{n-1} (T, \C )$ which proves
the desired statement.
\end{proof}

As a consequence we get a result first proven by Varolin  \cite{V1} .

\begin{corollary}\label{nc.11.23.10}
Any torus $X=(\C^*)^n$ has AVDP with respect to an invariant volume form $\omega$.
\end{corollary}

\begin{proof} Condition (B) from Theorem \ref{nc.07.14.10} is valid for $X$ by Lemma \ref{nc.07.15.15}.
Hence aside from a trivial case of $k=1$ it suffices to check    Condition (A) from Proposition \ref{nc.07.13.30}.
For $i \ne j$ the semi-simple vector fields $\nu_i$ and $\nu_j$ commute where $\nu_i =z_i \partial / \partial z_i$.
%Hence for any nonconstant Laurent polynomial $f$ in $z_j$
%Proposition \ref{nc.08.11.20} implies that $\nu_i' =\nu_j (f) \nu_i \in\LieAO (X)$ corresponds to an exact form.
Consideration of wedge products $\nu_i \wedge \nu_j$ implies     Condition (A) and the desired conclusion.
\end{proof}

\begin{notation}\label{nc.07.14.25} Recall that every reductive group $G$
is a complexification of a  maximal compact  subgroup $K$ of $G$, i.e
$G =K^\C$. Similarly if $R$ is a closed reductive subgroup of $G$ then $R = L^\C$ for a maximal compact subgroup $L$ of $R$. Extending $L$
to a maximal compact subgroup of $G$ (which is unique up to conjugation by the Cartan-Iwasawa-Maltsev theorem)
we can suppose that it is contained in $K$. We need the following Mostow Decomposition theorem (e.g.,
see \cite{Mo1}, \cite{Mo2}, and also  \cite{Hei}[Section 3.1] ).
\end{notation}

\begin{theorem}\label{nc.07.14.40} Let $K$ and $L$ be as before and $\lgoth$ (resp. $\lgoth^\C$)
be the Lie algebra of $K$ (resp. $K^\C$).
Suppose that $L$ acts on $\lgoth^\C$ by the adjoint representation. Then there exists a $L$-invariant
linear subspace $\mgoth$ of $\lgoth$ such that the map $K\times_L \sqrt{-1} \mgoth \to K^\C /L^\C$
given by $(k,v) \to k \cdot \exp {vL^\C}$ is an isomorphism of topological $K$-spaces.

\end{theorem}

\begin{corollary}\label{nc.07.14.50} The homogeneous space $X:=G/R = K^\C / L^C$  is homotopy equivalent
to $K/L$ and there is a natural isomorphism between $H_{1} (X, \C )$ and $H^{n-1} (X, \C )$.
In particular, when $G$ is semi-simple then $H^{n-1} (X, \C )=0$ and Condition (B) holds.
\end{corollary}

\begin{proof}
%Consider the family of homeomorphisms of $X \simeq K\times_L \sqrt{-1} \mgoth$ into itself
%given by the formula $(k, v) \to (k, tv)$ where $t$ is a real number running over the interval $[0,1]$.
%This shows that
Since  $X \simeq K\times_L \sqrt{-1} \mgoth$ is a vector bundle over $K/L$ we see that
$K/L$ is a retract of $X$ which is the first statement. Since $K/L$ is a real $n$-dimensional
compact manifold by the Poincare duality we have a natural isomorphism between
$H_1(K/L, \C )$ and $H^{n-1}(K/L , \C )$ which is the second statement. In the case of a semi-simple
$G$ the fundamental group $\pi_1 (K)$ is finite. Hence the exact homotopy sequence of the
locally trivial fibration $K \to K/L$ implies that $\pi_1 (X) \simeq \pi_1 (K/L)$ is finite.
Therefore $H_1(X, \C)=0$ and as a result $H^{n-1}(K/L , \C )=0$. Now Condition (B) holds automatically.

\end{proof}

\begin{notation}\label{nc.11.18.10} Let $G$ be a connected linear algebraic group, $R$ be a closed reductive subgroup of $G$
(in particular $X=G/R$ is affine by the Matsushima theorem \cite{Ma}), $G_0$ be a maximal reductive subgroup of $G$,
and $\cR_u$ be the normal subgroup of $G$ whose Lie algebra is the unipotent radical of the Lie algebra of $G$
(i.e. as an affine algebraic variety $\cR_u$ is isomorphic to a Euclidean space).
Then $G_0$ admits an unramified covering  $\hat G_0$ which is of form
$\hat G_0 \simeq \hat S \times \hat T$ with $\hat S$ being a simply connected semi-simple (Levi)
subgroup of the group $\hat G_0$ and a torus $\hat T$ being the connected component of the center of $\hat G_0$.

\end{notation}

\begin{lemma}\label{nc.07.28.10}  Let Notation \ref{nc.11.18.10} hold. Then $X=G/R$
can be presented as $X \simeq \hX /\Gamma$ where $\hX =Y \times \cR_u$ and $\Gamma$
is a finite subgroup of the center of $\hat G_0$ acting on
a $\hat G_0$-homogeneous space $Y$. Furthermore,
there is a $\hat G_0$-equivariant isomorphism  $\varphi : Y \to X_1 \times X_2$
where $X_1=\hat S /\hat R$ for some connected reductive  subgroup $\hat R$ of $\hat S$,
and $X_2$ is a subtorus in the connected component of the center of $\hat G_0$.

\end{lemma}

\begin{proof} Extend $R$ to a maximal reductive subgroup. Since such a subgroup is unique up to conjugation \cite{Mo}
we can suppose that $R$ is contained in $G_0$.
By Mostow's theorem \cite{Mo} $G$ can be viewed as a semi-direct product of $G_0$
and $\cR_u$. Hence as an affine algebraic variety $X$ is isomorphic to $(G_0/R) \times \cR_u$.
Taking $\hat R_0$ as the connected identity component of the preimage of $R$ in $\hat G_0$
we get the desired $\hX = Y \times  \cR_u$ such that $Y=\hat G_0/\hat R_0$ and $X = (Y/\Gamma ) \times \cR_u =\hat X /\Gamma $
for some finite subgroup $\Gamma$ in the center of $\hat G_0$.

To show existence of an isomorphism $\varphi$ consider
the image $\hat T_0$ of $\hat R_0\subset \hat S \times \hat T$ in $\hat T$ under the natural projection. Let $T$ be another subtorus of $\hat T$ such that
$\hat T$ is naturally isomorphic to $\hat T_0 \times T$. Then any element $g \in \hat G_0$ can
be presented as $st_0t$ where $s \in \hat S, t_0 \in \hat T_0$, and $t \in T$. Note that
for any coset $g\hat R_0$ one can choose a representative in the form $st$. Let $\hat R = \hat S \cap \hat R_0$
and consider the map $\varphi : \hat G_0/\hat R_0 \to \hat S/\hat R \times T$ given by $gR \to (s\hat R, t)$.
One can see that this map is well-defined and bijective. Furthermore if we define the action of $g'=s't_0't' \in \hat G_0$ on $(s\hat  R,t)$
by formula $(s\hat R , t) \to (s's\hat R , t't)$ then $\varphi$ is $\hat G_0$-equivariant.
It remains to check that $\hat R$ is reductive which is a consequence of the Matsushima theorem
since otherwise $\hat S/\hat R$ (and therefore $\hX$) is not affine. This concludes the proof.

\end{proof}

\begin{notation}\label{nc.11.18.30}
Let $X=X_1\times X_2$ and $p_i : X \to X_i$ be the natural projection. Suppose that $\omega_i$ is the volume form
on $X_i$. Then we denote the volume form $p_1^*(\omega_1) \wedge p_2^*(\omega_2)$ by $\omega_1 \times \omega_2$.
In a more general setting let $\tau_i$ be a $k_i$-form on $X_i$. Then the $(k_1+k_2)$-form $p_1^*(\tau_1) \wedge p_2^*(\tau_2)$
will  be also denoted by $\tau_1 \times \tau_2$.
\end{notation}

\begin{remark}\label{nc.12.15.30}
It is worth mentioning that  if $X_1$ does not admit nonconstant invertible regular functions then 
any volume form $\omega$ on $X$ can be presented as   $\omega =\omega_1 \times \omega_2$
provided each of factors admits an algebraic volume form. Since tori and Euclidean spaces admit
such forms Proposition \ref{nc.08.02.10} from Appendix implies that
any algebraic volume form $\hat \omega$ on $\hat X = X_1 \times X_2 \times \cR_u$ from Lemma \ref{nc.07.28.10}
can be presented as $\omega_1 \times \omega_2 \times \omega_{\cR_u}$ where $\omega_{\cR_u}$
is an algebraic volume form on $\cR_u$.  The absence of nonconstant  regular  invertible  functions
on $X_1$ and $\cR_u$ implies that  $\omega_1$ and $\omega_{\cR_u}$ are unique up to constant factors,
i.e. $\hat \omega$ is determined by the choice of $\omega_2$.
\end{remark}

\begin{proposition}\label{nc.11.18.20} Let the assumptions of Lemma \ref{nc.07.28.10} hold, $\omega_2$ as before be an invariant volume form
on the torus $X_2 \simeq T$, and $\omega_Y=\omega_1 \times \omega_2$. Then Condition (B) holds for the pair $(Y , \omega_Y )$. Furthermore, the space
$H^{n-1} (Y , \C)$ (where $n$ is the dimension of $Y$) is generated by the image of $\Theta (\IVF_{\omega_Y} (Y, \Gamma ))$
under the De Rham homomorphism $\Phi_{n-1}$  where
$\IVF_{\omega_Y} (Y, \Gamma ) \subset \IVF_{\omega_Y} (Y)$ consists of $\Gamma$-invariant vector fields.
In particular, Condition (B) is valid for the triple $(Y,  \omega_Y , \Gamma )$.

\end{proposition}

\begin{proof} By Lemma
 \ref{nc.07.15.15} and Corollary \ref{nc.07.14.50}  it suffices to consider the case
when each of the factors $X_1$ and $X_2$ is nontrivial.
Let $n_i$ be the dimension of $X_i$ and $n=n_1+n_2$. Then as we mentioned in the proof of Corollary \ref{nc.07.14.50}
$H^{n_1-1} (X_1, \C)=0$. Hence the K\"unneth formula implies that the space $H^{n-1}(X, \C )$ is generated by the images of
closed algebraic $(n-1)$-forms that can be presented as $\omega_1 \times \tau $ where $\tau$ is a closed algebraic
$(n_2-1)$-form on $X_2$. We saw in the proof of Lemma \ref{nc.07.15.15}
that such $\tau$ can be chosen in the form $\tau =\iota_\nu \omega_2$ where $\nu$ is a semi-simple vector field on the torus $T \simeq X_2$
associated with the multiplication by elements of a $\C^*$-subgroup $F$ of $T$. Furthermore, we saw that such $\nu$ is
of $\omega_2$-divergence zero. Hence  $H^{n-1} (Y , \C)$ is generated by $\Phi_{n-1}(\Theta (\IVF_{\omega_Y} (Y))$.
For the last statement it remains to note that since the actions of $\Gamma$ and $F$ (generated by multiplications)
commute such fields $\nu$ are $\Gamma$-invariant.
\end{proof}

Though we shall not use this fact later it is worth mentioning that Condition (B) is also valid for $\hX$ from Lemma \ref{nc.07.28.10}
by the following.

\begin{lemma}\label{nc.07.15.10} Suppose that $X=Y \times \C^k$ where $k \geq 1$
and  $\omega =\omega_Y \times \omega_{\C^k}$ be an algebraic volume form on $X$.
In the case of $k=1$ let the image of $\omega_Y$ under the De Rham homomorphism be a generator
of $H^{n-1}(Y, \C)$ where $n$ is the dimension of $X$ (by Corollary \ref{nc.11.30.10} in Appendix
this assumption is always true in the case when $Y$
is a connected homogeneous space without nonconstant invertible functions). Then Condition (B) is valid for the pair $(X , \omega )$.

\end{lemma}

\begin{proof}
By the K\"unneth formula $H^{n-1}(X, \C )=0$
as soon as $k\geq 2$ and Condition (B) holds. Let $k=1$ and $z$ be a coordinate on $\C^k$. Then by the K\"unneth
formula $\omega_1$ corresponds to a generator of $H^{n-1}(X, \C ) =\C$.
The derivative $\frac{\rm d}{{\rm d} z}$ induces a locally nilpotent derivation $\xi$ on $X$ which
is automatically divergence-free with respect to any volume form.
Note also that $\iota_\xi \omega =\omega_1$ which yields condition (B).
\end{proof}

\section{%Inclusion $\cB_{n-1} (X, \Gamma) \subset \Theta (\LieAO (X, \Gamma ))$
 Condition (A)  for homogeneous spaces}

 In order to apply Theorem \ref{nc.07.14.10} one needs to check  Conditions (A) and  (B). 
%one needs to show that  $\Theta (\LieAO (X, \Gamma ))$ contains $\cB_{n-1} (X, \Gamma)$ as in Proposition \ref{nc.07.13.30},
%e.g., to check      Condition (A). 
For homogeneous spaces of semi-simple groups  Condition (A)  follows from the next lemma
and Proposition \ref{nc.11.16.20}. 

\begin{lemma}\label{nc.08.10.10} Let $X_1$ be a homogeneous space of a semi-simple group $\hat S$, $\Gamma_1$ be a finite
subgroup of the center of $\hat S$, and    $v_1,  v_2 $ be non-collinear vectors in the space $T_{x_1}X_1$
at some point $x_1 \in X_1$.
Suppose that $\cN_0$ is the set of locally nilpotent vector fields on $X_1$ associated with multiplications by $\C_+$-subgroups of $\hat S$
and  that $H$ is  the group of algebraic automorphisms of $X$ generated by elements of phase flows associated with elements from
the set $\cN$ of all locally nilpotent vector fields of form $f\xi$ where $\xi \in \cN_0$ and the function $f \in \Ker \xi$ is $\Gamma_1$-invariant.

Then the orbit $O$ of $v_1 \wedge v_2$ under the action of the isotropy group $H_{x_1}$ generates the whole wedge-product space $V= \Lambda^2 T_{x_1} X_1$.
\end{lemma}

\begin{proof} Let $\nu$ be a complete algebraic vector field and $f \in \Ker \nu$ be a function such that
$f(x_1)=0$. Then the phase flow $\varphi_t$ associated with $f\nu$ generates an isomorphism   $T_{x_1}X_1 \to T_{x_1}X_1$
given by the formula $$w \to w +t {\rm d} f (w) v \leqno{(3)}$$ where $v= \nu (x_1)$. Note that each $\xi \in \cN_0$ corresponds
a nilpotent element of the Lie algebra of $\hat S$.  Hence the set $\{ \xi (x) | \xi \in \cN_0 \}$ generates $T_{x}X_1$ for every $x \in X_1$ which implies that
$X_1$ is an $H$-flexible variety in terminology of \cite{AFKKZ}. That is, $H$ acts
transitively on $X_1$. In particular replacing if necessary $x_1$ by $h(x_1)$ for some $h \in H$ we can
treat $x_1$ as a general point of $X_1$.
Assume that $\nu = h_* (\nu_0)$ for some $\nu_0 \in \cN_0$. Since it is enough to consider the case of
$\dim X_1 \geq 3$, we can suppose that $v_2$ is not a linear combination of $v_1$ and $v=\nu (x_1)$.  This means that if
$\rho : X_1 \to Q ={\rm Spec} \Ker \nu$ is the quotient morphism then $u_1=\rho_* (v_1)$ and $u_2=\rho_*(v_2)$ are not collinear, i.e. for
a general regular function $f$ on $Q$ one has ${\rm d} f (u_1) \ne 0$ and  ${\rm d} f (u_2) = 0$.  Furthermore, by construction the $H$-action
on $X_1$ commutes with the $\Gamma_1$-action which yields a $\Gamma_1$-action on $Q$. Thus we can choose $f$ to be $\Gamma_1$-invariant, i.e.
$f \nu \in \cN$. Treating $f$ as a function on $X_1$
we get  ${\rm d} f (v_1) \ne 0$ and  ${\rm d} f (v_2) = 0$. In combination with the formula (3) this implies that the span of $O$ contains
the wedge product $v_1 \wedge v$.  Now choose another locally nilpotent derivation for which the value $u$ at $x_1$
is not a linear combination of $v_1$ and $v$. Repeating the argument as before we see that the space of $O$ contains $u \wedge v$.
By \cite{AFKKZ}[Corollary 4.3] $u$ and $v$ can be chosen sufficiently general (more precisely, the set $\{h_*( \nu )(x_1) | \nu \in \cN_0 , h \in H \}$ coincides with $T_{x_1}X_1$).
Hence  we get the desired conclusion.

\end{proof}

\begin{remark}\label{nc.01.13.10}    Let ${\rm SAut} (X)$ be the subgroup of the group $\Aut (X)$ of algebraic automorphisms of $X$ generated
by elements of all algebraic one-parameter unipotent subgroups of $\Aut (X)$. Recall that according to one of equivalent
definitions \cite{AFKKZ} a normal affine
algebraic variety $X$ is flexible\footnote{The class of flexible varieties includes homogeneous spaces of extensions of semi-simple groups by
unipotent radicals, non-degenerate toric varieties, cones over flag varieties and del Pezzo surfaces (of degree at least 4), hypersurfaces of form $\{ uv=p (\bar x ) \} \subset \C^{n+2}_{u,v, \bar x}$,
homogeneous Gizatullin surfaces (except for $\C^* \times \C$), etc..}
if ${\rm SAut} (X)$ acts transitively on the smooth part of $X$.
In this terminology one can have the following straightforward extension of Lemma \ref{nc.08.10.10}: for every smooth point $x \in X$ the isotropy
group $({\rm SAut } (X) )_x$ induces an irreducible action on $\Lambda^2 TX$.
In the presence of a semi-compatible pair this yields   Condition (A).
Thus by Theorem \ref{nc.07.14.10} we have the following fact which will not be used later.
\end{remark}

\begin{theorem}\label{nc.01.13.20} Let $X$ be a smooth flexible variety equipped with an algebraic volume form $\omega$
such that $H^{n-1}(X, \C)=0$ where $n=\dim X$. Suppose that $X$ admits a semi-compatible pair of divergence-free
vector fields. Then $X$ has AVDP.

\end{theorem}

Taking into consideration Example \ref{nc.11.16.00} we have the following.

\begin{corollary}\label{nc.01.13.30} Let $X$ be a smooth flexible variety equipped with an algebraic volume form $\omega$
such that $H^{n-1}(X, \C)=0$ where $n=\dim X$. Suppose that either $X$ admits a fixed point free non-degenerate
algebraic $SL_2$-action or $X=X_1 \times X_2$ where $X_i$ is a flexible variety of dimension at least 1
equipped with a volume form $\omega_i$ such that $\omega =\omega_1 \times \omega_2$. Then $X$ has AVDP.

\end{corollary}

\begin{notation}\label{nc.11.21.30}
In the rest of this section we consider a situation when  $X$
is isomorphic to the direct product $X_1 \times X_2$ of $n_i$-dimensional   $\Gamma_i$-varieties  $X_i$ equipped with  multi-volume forms $\omega_i$
(in the sense of Definition \ref{nc.11.20.10}(3)) 
 such that $\chi_{\omega_i}$ is a character for $i=1,2$.  We suppose also that these $\Gamma_i$-actions are {\bf free}.  
Then we 
equip $X$ with the  multi-volume  form $\omega =\omega_1 \times \omega_2$ and treat it as an $n$-dimensional
 $\Gamma$-variety where $\Gamma$ is a subgroup
of $\Gamma_1\times \Gamma_2$ acting naturally on $X$.
For each point $\bar x =(x_1,x_2) \in X$ we have the natural embedding $T_{x_i} X_i \hookrightarrow T_{\bar x} X$ such that
$T_{\bar x}X \simeq T_{x_1} X_1 \oplus T_{x_2}X_2$ which enables us to treat $T_{x_i}X_i$ (resp.  $\Lambda^kT_{x_i}X_i$) as a subspace of $T_{\bar x} X$
(resp. $\Lambda^k T_{\bar x} X$).  We 
denote by $V$ the smallest subbundle of $\Lambda^2TX$ that contains all wedge-products of form $v_1 \wedge v_2$ where
$v_i \in T_{x_i}X_i \subset T_{\bar x}X$ with $\bar x =(x_1,x_2)$ running over $X$.
Furthermore, for every
vector field $\xi$ on $X_i$ the embedding  $T_{x_i} X_i \hookrightarrow T_{\bar x} X$ yields
an induced vector field on $X$ which will be denoted by $\xi'$. We treat also $\C [X_i]$ as a natural subring of $\C [X]$.
\end{notation}

\begin{proposition}\label{nc.11.25.10}
Let Notation \ref{nc.11.21.30} hold, each $X_i$  has $\Gamma_i$-AVDP with respect to $\omega_i$, and let $\Gamma = \Gamma_1 \times \Gamma_2$. Then $X$ has $\Gamma$-AVDP with respect to $\omega$.
\end{proposition}

 \begin{proof}
Replacing $X_i$ by  $X_i/\Gamma_i$ we see that it is enough to establish the fact for trivial  $\Gamma_1$ and $\Gamma_2$.
This was proven in the case of volume forms $\omega_1$ and $\omega_2$ in  \cite[Proposition 4.3]{KaKu4}
and the proof works without change for multi-volume forms (as well as the proof of the preceding  \cite[Lemma 4.2]{KaKu4}).
\end{proof}
%In \cite{KaKu4}[Proposition 4.3] this fact was proven for trivial $\Gamma_1$ and $\Gamma_2$ but the proof remains valid in the general case
%with straightforward adjustments (more precisely, instead of the rings of regular function $\C [X_i]$
%one need to consider the subrings of $\Gamma_i$-invariant functions $\C [X_i , \Gamma_i]$). Hence we omit it and 

Now we  concentrate on
a more difficult case of $\Gamma \ne \Gamma_1 \times \Gamma_2$.
      
\begin{lemma}\label{nc.11.20.20a}  Let  Notation \ref{nc.11.21.30} hold
and let %$\{ \eta_{ij} \}_{j=1}^k \subset \IVF_{\omega_i} (X_i, \Gamma_i)$
 $\{ \xi_{ij} , \eta_{ij} ) \}_{j=1}^k \subset \IVF_{\omega_i} (X_i, \Gamma_i)$  be a collection of   semi-compatible pairs
 either  satisfying at least one of assumptions (i) and (ii) from Proposition \ref{nc.08.13.10} or being a pair semi-simple fields
satisfying the assumptions of Proposition \ref{nc.04.05.13.3}.  Let also
 each $\xi_{1j}$ be locally nilpotent and
$\xi_{2j}$  be either locally nilpotent or semi-simple.
Suppose additionally 
that  at a general point $x_i \in X_i, \, i =1,2$

 {\rm (1)}  the set $\{ \xi_{ij} (x_i)  \}_{j=1}^k$ generates $T_{x_i}X_i$ and

 {\rm (2)}  the set $\{ \xi_{ij}(x_i) \wedge \eta_{ij}(x_i) \}_{j=1}^k$
generates the whole wedge-product space $\Lambda^2 T_{x_i}X_i$.

Let a group $F$ of algebraic automorphisms
act transitively on $X$ so that
this action commutes with the $\Gamma$-action and up to constant factors preserves $\omega$.
Then $\Theta (\LieAO (X, \Gamma ))$ contains $\cB_{n-1} (X, \Gamma)$.

\end{lemma}

\begin{proof}
Let us consider vector fields  $\xi_{ij}'$ (resp. $\eta_{ij}'$) as in Notation \ref{nc.11.21.30}.
Note that each pair $( \xi_{ij}' , \eta_{ij}' )$ is $\Gamma$-semi-compatible   by Propositions \ref{nc.08.13.10}  and \ref{nc.04.05.13.3}  
(the smoothness of the quotient morphism required in  Proposition \ref{nc.04.05.13.3} survives factorization with respect to $\Gamma$
since the $\Gamma$-action is free). 
 By the same reason  every pair $\{ \xi_{1j}' , \xi_{2l}')$ is also $\Gamma$-semi-compatible.
Then for the general point $\bar x =(x_1,x_2) \in X$ the set $\{ \xi_{1j}'(\bar x ) \wedge \eta_{1j}'(\bar x) \}_{j=1}^k$
generates the image of the subspace $\Lambda^2 T_{x_1}X_1$ in $\Lambda^2 T_{\bar x}X$ under the natural embedding. Similarly one can take care of $\Lambda^2 T_{x_2}X_2\subset
\Lambda^2 T_{\bar x}X$.
By assumption  (1)  the fiber $V_{\bar x}$ of $V$  from Notation \ref{nc.11.21.30}  is generated by elements of form  $\xi_{1j}'(\bar x ) \wedge \xi_{2l}'(\bar x)$.
Since $\Lambda^2 T_{x_1}X_1 + \Lambda^2 T_{x_2}X_2 + V_{\bar x} =\Lambda^2 T_{\bar x}X$  we are under the assumptions of Corollary \ref{nc.01.01.10},
%i.e. $\Theta (\LieAO (X, \Gamma ))\supset D (\cC_{n-1} (X, \Gamma ))$.
%By Proposition \ref{nc.08.13.40}
%one has  also $\cB_{n-1} (X, \Gamma ) =D (\cC_{n-1} (X, \Gamma ))$ which
 which implies Condition (A) and thus the inclusion $\Theta (\LieAO (X, \Gamma ))\supset \cB_{n-1} (X, \Gamma)$.
\end{proof} 

\begin{remark}\label{nc.11.21.10}
Note that  with obvious modifications  Lemma \ref{nc.11.20.20a} is valid even if one of
factors, say, $X_2$ is one-dimensional. That is,  one can consider $X_2$ equal to $\C_z$ (resp. $\C^*$) with $\omega_2 ={\rm d} z$
(resp.  $\omega_2 =\frac {{\rm d}z} {z}$).
\end{remark}

\section{Main Theorem}

The aim of this section is the following.

\begin{theorem}\label{nc.11.23.20} Let Notation \ref{nc.11.18.10} hold, i.e.  $R$
is a closed reductive subgroup of a linear algebraic group $G$, $X=G/R$ be the homogeneous space of left cosets
which by Lemma \ref{nc.07.28.10}  can be presented as $X \simeq (Y/\Gamma ) \times \cR_u$
where $\cR_u$ is isomorphic to a Euclidean space
%normal subgroup of $G$ associated with the unipotent radical of the Lie algebra of $G$,
and $Y$ is a homogeneous space of a reductive group $\hat G_0$ that is an unramified covering of a maximal reductive subgroup $G_0$
of $G$ and has the following properties

(i)  a Levi semi-simple subgroup $\hat S$ of $\hat G_0$ is simply connected;

(ii) $Y=X_1 \times X_2$ where $X_1=\hat S /\hat R$ for some connected reductive  subgroup $\hat R$ of $\hat S$,
and $X_2$ is a subtorus in the connected component of the center of $\hat G_0$;

(iii) $\Gamma$ is a finite subgroup of the center of $\hat G_0$.

Let  $\hat \omega =\omega_1 \times \omega_2\times \omega_{\cR_u}$ be an algebraic volume form on $\hX=X_1 \times X_2 \times \cR_u$
as in Remark \ref{nc.12.15.30}. Suppose that $\omega_2$ is an invariant volume form on the torus $X_2$.

Then $\hat X$ has $\Gamma$-AVDP (with respect to $\hat \omega$). %, $\LieAO (\hX , \Gamma )= \AVF_\omega (\hX , \Gamma )$.
\end{theorem}

%The proof of this theorem requires some preparations. Since $\cR_u$ is a Euclidean space which always has
%AVDP (with respect to any volume form since such a form is
%unique up to a constant factor)  we can suppose that the unipotent radical $\cR_u$ is trivial
%because of the next result.

\begin{remark}\label{nc.08.13.20} It is worth mentioning that we do not assume existence of a volume form on $X= \hat X/\Gamma$ in Theorem \ref{nc.07.14.10}.
However we have AVDP for $\hX/\Gamma$ with respect to the multi-volume form induced by $\omega$ in the sense of Definition \ref{nc.11.20.10}.
\end{remark}

\begin{corollary}\label{nc.12.15.20} Let  $G$ be a linear
algebraic group, $R$ be a closed reductive subgroup of $G$, and $X$ be the homogeneous space $G/R$ . Suppose that $X$ has a $G$-invariant algebraic volume form
$\omega$. Then $X$ has AVDP with respect to $\omega$. \end{corollary}

\begin{proof}
Let $\hX=X_1 \times X_2 \times \cR_u$ be from Theorem \ref{nc.11.23.20} and let
$\hat \omega=\omega_1 \times \omega_2 \times \omega_{\cR_u}$ be an algebraic volume form on $\hX$ as in Notation \ref{nc.12.15.30}.
Recall that $\hat G_0$ acts naturally on $\hat X$ so that the action of $t \in T \simeq X_2 \subset \hat G_0$ on
$x=(x_1,x_2,r) \in \hX$ is given by $t.x=(x_1, tx_2,r)$ (see the proof of Lemma \ref{nc.07.28.10}). In particular,
$\hat \omega$ is $T$-invariant iff $\omega_2$ is.
When $\hat \omega $ is induced by $\omega$  it must be $T$-invariant. Hence the assumptions of Theorem   \ref{nc.11.23.20}  hold,
i.e. $\Lie_{\rm alg}^{\hat \omega} (\hat X, \Gamma ) =\AVF_{\hat \omega} (\hat X, \Gamma)$.
Now the natural isomorphisms   $\Lie_{\rm alg}^{\hat \omega} (\hat X, \Gamma ) \simeq \LieAO (X )$
and $\AVF_{\hat \omega} (\hat X, \Gamma)\simeq  \AVF_{\omega} (X)$ imply the desired conclusion.
\end{proof}

Before presenting the proof let us consider one case related to two-dimensional homogeneous spaces where the argument is
very specific since  Proposition \ref{nc.11.16.20} is not  applicable in dimension 2  and therefore the technique of semi-compatible fields
does not work.

\begin{example}\label{nc.12.29.10}
Consider $G=SL_2$ as a subvariety of $\C^4_{a_1,a_2,b_1,b_2}$  given by $a_1b_2-a_2b_1=1$,
i.e., matrices
\begin{center}$ A =  \left[ \begin{array}{rr}
a_1& a_2  \\
b_1 & b_2 \\
 \end{array}  \right]$\end{center}
 are elements of $G$.
Let $T\simeq \C^*$ be the torus consisting of the diagonal elements and $N$ be the normalizer of $T$
in $SL_2$. That is, $N/T \simeq \Z_2 =\Gamma_1$ where the matrix
$$\, \, \, \, \, A_0= \left[
\begin{array}{rr}
0 & -1  \\
1 & 0 \\
\end{array}  \right] \in N$$ generates the nontrivial coset of $N/T$.

By Proposition \ref{nc.08.02.10} the homogeneous space $X_1=G/T$ possess a $G$-invariant volume form
$\omega_1$ while $G/N \simeq X_1/\Gamma_1$ does not (see Example \ref{nc.08.01.30})).  It is worth mentioning that  $X_1$ 
has AVDP since it
is isomorphic to the hypersurface $uv=x^2-1$ in $\C^3_{u,v,x}$ and such hypersurfaces were dealt with in \cite{KaKu4}.
Also the action of $\Gamma_1$ on $X_1$
is given by $(u,v,x)\to (-u,-v,-x)$ \footnote{Indeed,
ring of $T$-invariant regular functions on $G$ is generated by $u=a_1b_1, v=a_2b_2, y=a_1b_2$, and $z=a_2b_1$
where $y=z+1$. Hence $X_1$ is isomorphic to the hypersurface $uv=z(z+1)$ in $\C^3_{u,v,z}$. Let $x=z+1/2$. Then
 $X_1$ is isomorphic to $uv=x^2-1/4$ in $\C^3_{u,v,x}$ and replacing $(u,v,x)$ by $(2u,2v,2x)$ we get the desired equation. The formula for
the $\Z_2$-action (induced by multiplication by $A_0$) is now a straightforward computation.}.

Furthermore $\Gamma_1$-AVDP for $X_1$ was established in \cite{Leu}.   
Let us show $\Z_2$-AVDP for a more complicated object: $X=X_1 \times X_2$ where $X_2=\C^*_z$ and  
the action of $\Gamma \simeq \Z_2$ is given by $(u,v,x,z)\to (-u,-v,-x,-z)$. That is, we have the group $F = SL_2 \times \C^*$ acting
naturally on $X$ so that the action commutes with $\Gamma$-action and preserves the volume form $\omega = \omega_1 \times \omega_2$
(where $\omega_2$ is an invariant volume form on the torus $X_2$).

Let $\Gamma_2 \simeq \Z_2$ act on
$X_2 \simeq \C^*_z$ by $z \to -z$. Since for $i=1,2$ the variety $X_i$ has $\Gamma_i$-AVDP there are vector fields
 $\{ \xi_{ij} \}_{j=1}^k \subset \IVF_{\omega_i} (X_i, \Gamma_i )$
such that the set $\{ \xi_{ij} (x_i)  \}_{j=1}^k$ generates $T_{x_i}X_i$ at  each point $x_i \in X_i$. 
 Furthermore, by Proposition \ref{nc.08.13.10} one can suppose that every pair of vector fields on $X$ induced by $\xi_{1i}$ and $\xi_{2j}$ 
is $\Gamma$-semi-compatible. 
That is, by Corollary \ref{nc.01.01.10} $\Theta (\LieAO (X, \Gamma ))$ contains $D (\cC_{n-2}^V (X, \Gamma ))$ where 
$V$ is as in Notation \ref{nc.11.21.30}.

In our particular case $\chi_\omega$ is the nontrivial character on $\Gamma \simeq \Z_2$ and hence any element of $\cZ_{n-1}(X, \Gamma )$ is of form 
$$\sum_{k=-\infty}^{\infty} (f_k\omega_1 \times z^k + \tau_k \times (z^k \omega_2))$$
$\tau_k$ is a $1$-form on $X_1$ which is $\Gamma_1$-invariant (resp. $\Gamma_1$-anti-invariant)
when $k$ is odd (resp. even), and $f_k \in \C [X_1]$ is $\Gamma_1$-invariant (resp. $\Gamma_1$-anti-invariant)
when $k$ is even (resp. odd) because $\omega_2= \frac{{\rm d} z}{z}$ is $\Gamma_2$-invariant and $\omega_1$ is $\Gamma_1$-anti-invariant. Similarly
$\cC_{n-2}^V(X, \Gamma )$ is generated (as a vector space) by elements of form
$\tau_k \times z^k$. Note that $$D( \tau_k \times z^k )= {\rm d} \tau_k \times z^k+k\tau_k \times (z^k \omega_2) .$$
Since $D( \tau_k \times z^k )  \subset \Theta (\LieAO (X, \Gamma ))$,
in order to establish  $\cZ_{n-1} (X, \Gamma) \subset \Theta (\LieAO (X, \Gamma ))$ we
need to check that  exact forms of type
$$ \tau_0 \times \omega_2 + \sum_{k=-\infty}^{\infty} f_k\omega_1 \times z^k $$ are contained in
 $\Theta (\LieAO ( X , \Gamma ))$. Since this form is closed one can see that $f_k=0$ for every $k \ne 0$,
 and we are left with a form $\tau_0 \times \omega_2 +f_0\omega_1 \times 1$. Note that $\tau_0$ is
 a closed form. %\footnote{That is, this form is exact since $H^1(X_1, \Z)=0$.} on $X_1$ since $D(f_0\omega_1 \times 1)=0$. 
 By   \cite[Theorem 3.6 and Remark 3.7]{Leu} such an anti-invariant $\tau_0$ belongs to the span of
$\Theta (\IVF_{\omega_1} (X_1, \Gamma_1))$, i.e.  $\tau_0 \times \omega_2 \in \Theta (\LieAO ( X , \Gamma ))$.
Note also that $f_0\omega_1 \times 1= \iota_\nu \omega$ where $\nu$ is the complete $\Gamma$-invariant field $f_0z \partial /\partial z$, we are done.

\end{example}

\subsection{Proof of Theorem \ref{nc.11.23.20} }
Since $\cR_u$ is a Euclidean space which always has
AVDP (with respect to any volume form since such a form is
unique up to a constant factor)  we can suppose that the unipotent radical $\cR_u$ is trivial
because of Proposition \ref{nc.11.25.10}.

Thus from now on $G$ is reductive.
That is, $X =\hX /\Gamma$ where $\hX =X_1 \times X_2 =\hat G_0/\Gamma$ and $X_1$ can be assumed nontrivial by virtue of
 Corollary \ref{nc.11.23.10} .  Since $\Gamma$ is a subgroup of the center of   $\hat G_0$ its action
commutes with the actions of $\C_+$ and $\C^*$-subgroups of $\hat G_0$ induced by the left multiplication.
In particular, the semi-simple vector fields $\nu_i$ that appeared in the proof of
Corollary \ref{nc.11.23.10} are $\Gamma$-invariant. One can suppose
that each $\nu_i$ is tangent to $X_2$.
Since $\omega_2$ is an invariant volume form on $X_2$ these vector fields are of zero divergence with respect to this form.
Thus we have a collection of divergence-free $\Gamma$-invariant  semi-simple  vector fields $\{ \xi_{2j}, \eta_{2j}  \}_{j=1}^k$ on $X_2$ that
commute and for which  $\{ \xi_{2j}(x_2) \wedge \eta_{2j}(x_2) \}_{j=1}^k$
generates the whole wedge-product space $\Lambda^2 T_{x_2}X_2$ at any point $x_2 \in X_2$ as required in Lemma \ref{nc.11.20.20a}.\\

 {\em Case of $\dim X_1 \geq 3$.} 
 By Propositions \ref{nc.11.16.10} and \ref{nc.11.16.20} there is a semi-compatible pair $(\xi  , \eta )$ of locally nilpotent vector fields
on $X_1$.  Let $v_1$ and $v_2\in T_{x_1}X_1$ be the values of these vector fields at some general point $x_1 \in X_1$ (in particular these
vectors are not collinear).  Let
$\Gamma_1$ be the image of $\Gamma$  under the natural projection  $\hat G_0 = \hat S\times \hat T \to \hat S$
(i.e. $\Gamma_1$ is a finite subgroup of the center of $\hat S$). Suppose that $H$ is the group of automorphisms
of $X_1$ described in Lemma \ref{nc.08.10.10} (in particular, being generated by elements of $\C_+$-actions $H$ preserves any algebraic
volume form). Then the $H_{x_1}$-orbit of $v_1 \wedge v_2$ generates  $\Lambda^2 T_{x_1}X_1$.
Thus we get  a collection of divergence-free vector fields $\{ \xi_{1j}, \eta_{1j}  \}_{j=1}^k$ on $X_1$ such that $\{ \xi_{1j}(x_1) \wedge \eta_{1j}(x_1) \}_{j=1}^k$
generates the whole wedge-product space $\Lambda^2 T_{x_1}X_1$ at $x_1 \in X_1$   as required in Lemma \ref{nc.11.20.20a}. 

 Let the fields $\xi_{ij}', \eta_{ij}'$ on $X_1 \times X_2$  have the same meaning as
in Notation \ref{nc.11.21.30}. By construction they are $\Gamma$-invariant and one can suppose that $\{ \xi_{ij} \}_{j=1}^k$
generate $T_{x_i}X_i$ at every $x_i \in X_i$.  Since $\Gamma \subset \hat G_0$  Proposition \ref{nc.02.24.13.3} implies that  the $\Gamma$-action (as well as $\hat G_0$-action)
preserves $\omega$ up to a constant factor.
Hence Lemmas \ref{nc.08.13.40} and \ref{nc.11.20.20a} are applicable which 
implies that $\Theta (\LieAO (\hat X , \Gamma )$ contains $\cB_{n-1} (\hat X, \Gamma )$.   By Proposition \ref{nc.11.18.20}
Condition (B) holds for $\hat X$. Hence Theorem \ref{nc.07.14.10} yields the desired conclusion when $\dim X_1 \geq 3$.\\

 {\em Case of $\dim X_1 =2$.}   Choose  a subgroup $\hat S' \simeq SL_2$ of $\hat S$ so that it is
not contained in $\hat R$. Furthermore,  by the Cartan-Iwasawa-Maltsev theorem we can organize this choice so that maximal compact subgroups
$\hS_\R'$ and $\hR_\R$ of $\hS'$ and $R$ respectively are contained in the same maximal compact subgroup $\hS_\R$ of $\hS$. Note that
the group $\hS_\R' \cap \hR_\R$ is of dimension 1 since otherwise the complexification of  $\hS_\R' \cap \hR_\R$ is a two-dimensional reductive subgroup
of $\hS'$ but the only two-dimensional reductive group $\C^* \times \C^*$ is not contained in $SL_2$. Hence the image of $\hS_\R'$ under
the quotient morphism $\hS_\R \to \hS_\R / \hR_\R$ is  two-dimensional and therefore surjective. Hence  the image of $\hS'$ under
the quotient morphism $\hS \to \hS/ \hR$ is also surjective. Thus we can suppose that $\hS =SL_2$.
Since maximal tori are the only proper connected reductive subgroups of $SL_2$ we have $\hR \simeq \C^*$ and $X_1\simeq SL_2/\C^*$.

Hence $X_1$ possesses an algebraic volume form (see, Proposition \ref{nc.08.02.10}) and also AVDP \cite{Leu}.
Let $\Gamma_1$ be as before (i.e. $\Gamma_1$ is at most $\Z_2$ since it is a subgroup of the center of $SL_2$)
and let $\Gamma_2$ be the image of $\Gamma$ under the natural projection $\hat G_0 = \hat S\times \hat T \to \hat T$.
Then $\Gamma$ may be viewed as a subgroup of $\Gamma_1 \times \Gamma_2$. If they coincide then we are done by
Proposition \ref{nc.11.25.10} and the result of \cite{Leu} about $\Z_2$-AVDP for $X_1$. If not  
we can suppose that $\Gamma$ is naturally isomorphic to $\Gamma_1\simeq \Z_2$ and to $\Gamma_2$. % by Lemma \ref{nc.04.03.13.1a} 
%then
%$\Gamma =\Gamma_1' \oplus\Gamma_2'$ where $\Gamma_1'=\Z_2$  and $\Gamma_2'$ is the kernel of the natural homomorphism $\hat G_0 \to \hat S$.
%Hence replacing the torus $X_2=T$ by $T/\Gamma_2'$ we can suppose that
%$\Gamma =\Z_2$. 
Furthermore, one can present $T$ now as $T=T_1\times T_2$ where $T_1\simeq \C^*$ contains the generator of $\Gamma_2=\Z_2$
and $T_2$ is another torus. Hence $X$ is the product of $(X_1 \times T_1)/\Gamma$ and $T_2$ and by virtue of Proposition \ref{nc.11.25.10}
we can suppose now that $X_2 =T_1$. Now the desired conclusion follows from Example \ref{nc.12.29.10}.
\hspace{14.1cm}  $\square$

\section{Surfaces $p(x)+q(y)+xyz=1$}\label{surface}

Theorem \ref{nc.11.23.20}  is not the only application of the basic idea behind  our criterion.
In the case of a smooth affine simply connected surface $S$ equipped with an algebraic
volume form $\omega$ Condition (B) is trivial since $H^1(S, \C )=0$.
Furthermore, by the Grothendieck theorem  \cite{Gro}  for every  $f \in \C [S ]$ there is $ \xi \in \AVF_\omega (S) $ for which  $\iota_\xi \omega = {\rm d} f$
and the equality $\Theta (\LieAO (S))=\cB_1 (S)$ (which
implies $\LieAO (S)=\AVF_\omega (S)$) becomes equivalent
to the fact that such a $\xi$ can be chosen in $ \LieAO (S)$. The next
technical fact is useful for verification of this condition.

\begin{proposition}\label{nc.07.12.10} Let $S$ be a smooth affine surface equipped with an algebraic volume form $\omega$
and $\xi \in \AVF_\omega (S)$ be nonzero. Suppose that $f$
is a regular function such that $\iota_\xi \omega = {\rm d} f$.
Then $L_\xi (f)=0$ \footnote{This statement remains valid when $S$ is a complex surface, and $\omega , \xi$, and $f$ are holomorphic.}.
Furthermore, suppose that $S$ is rational, there are no nonconstant invertible functions on $S$,  and $\xi$ does not vanish identically on any divisor in $S$.
Then the kernel of $\xi$ in $\C [S]$ coincides with $\Ker \xi = \C [f]$.
\end{proposition}

\begin{proof} By Formula (2) one has
$$L_\xi (f)=\iota_\xi {\rm d} f+ {\rm d} \iota_\xi f =\iota_\xi {\rm d}f= \iota_\xi \iota_\xi \omega=0$$
which yields the first statement. For the second statement note that  $\Ker \xi$
is of transcendence degree 1 over $\C$. Indeed, $\Ker \xi \ne \C$ since $f$ is not constant and $\Ker \xi$ cannot be
of  transcendence degree 2 since otherwise being algebraically closed in $\C [S]$ it coincides with $\C [S]$.
Thus for any $g \in \Ker \xi$ the image of the map $(f,g) : S \to \C^2$ is a curve $C$. Since $S$ is rational $C$
is rational. Furthermore $C$ does not admit nonconstant invertible functions.
Hence $C$ is a polynomial curve, i.e. the ring of regular functions on its normalization is isomorphic to $\C [h]$
where $h$ is a rational continuous function on $C$. In particular $h$ generates a continuous rational function on $S$
(denoted by the same symbol $h$) which is regular because of the smoothness of $S$.  Suppose that $f =p (h)$ where
$p$ is a polynomial of degree at least 2. 
Then  $\iota_\xi \omega = {\rm d} f = p'(h) {\rm d} h$ where the left-hand side does not vanish on any divisor of $S$
while the right-hand side has a nontrivial zero locus $p'(h)=0$ which is a divisor. This contradiction concludes the proof.

\end{proof}

\begin{notation}\label{nc.12.16.10} We consider a hypersurface $S$ in $\C_{x,y,z}^3$ given by an equation
$$p(x)+q(y)+xyz=1, \, \, \, {\rm i.e.} \, \, \, z={\frac{1-p(x)-q(y)}{xy}}$$
where $p$ and $q$ are polynomials such that $p(0)=q(0)=0$ and
$1-p(x)$ and $1-q(y)$ have simple roots only.  Note that $S$ contains the torus $T \simeq \C_x^* \times \C_y^*$ and
up to a constant factor $$\omega = \frac{{\rm d}x \wedge {\rm d} y}{xy}$$ is the only algebraic volume form on $T$ that extends
regularly to $S$.

\end{notation}

\begin{remark} One can check that $S$ is obtained via half-locus attachments (in terminology of Fujita \cite{Fu}) to the
boundary of $T $ at points $(x_1,0), \ldots, (x_k,0)$ and  at points
$(0,y_1), \ldots (0,y_l)$ where $x_1, \ldots , x_k$  (resp. $y_1, \ldots , y_l$) are the roots
of $1-p$ (resp. $1-q$). That is, we blow $\C^2_{x,y}$ up at these points and remove the proper
transform of the cross $xy=0$.  In particular $S$ is simply connected, has no nonconstant invertible functions, and  is  of logarithmic Kodaira dimension 0
 since this dimension for torus is 0  and  half-locus attachments do not change the Kodaira dimension . The last fact implies that
$S$ does not admit nontrivial algebraic $\C_+$-actions and it can be also shown that it does not have nontrivial
algebraic $\C^*$-actions either.  Nevertheless $S$ is transitive with respect
to the group of holomorphic automorphisms  generated by elements of phase flows of complete algebraic vector fields
(it is enough to use the algebraic vector fields listed in Lemma \ref{nc.12.16.20} below).
Hence it is interesting to find out whether it has ADP or AVDP (with respect to $\omega$).
For the special case of $p(x)=x$ and $q(y)=y$
we described in \cite {KaKu3} all complete algebraic vector fields on $S$ which turned out to be divergence-free, i.e. ADP does not hold.
However as we see below AVDP is valid even in the general case.
\end{remark}

\begin{lemma} \label{nc.12.16.20} Every regular function $f$ on $S$ can uniquely be written in the form
$$f= a_0 + \sum_{i=1}^N a_i x^i + \sum_{i=1}^N b_i y^i +  \sum_{i=1}^N c_i z^i  +
     \sum_{i,j =1}^N a_{ij} x^i  y^j + \sum_{i,j =1}^N b_{ij} x^i  z^j+ \sum_{i,j =1}^N c_{ij} y^i  z^j \leqno{(4)}$$
and the vector fields
$$\delta_z= (q'(y)+xz) \partial / \partial x - (p'(x)+yz) \partial / \partial y,$$
$$\delta_y=-xy \partial / \partial x + ((p'(x) +yz) \partial / \partial z,$$
$$ {\rm and}\, \, \,  \delta_x= -xy \partial / \partial y + (q'(y) +xz) \partial / \partial z$$
are complete on $S$ and of $\omega$-divergence zero. Furthermore,
$\delta_z , \delta_y, \delta_x$ vanish on a finite set only and their kernels are $\C [z], \C [y]$,
and $\C [x]$ respectively.

\end{lemma}

\begin{proof} The first statement is a consequence of the equation $p(x)+q(y)+xyz=1$. For $\delta_x$ and $\delta_y$ all claims follow from the fact that these fields are the images
of the fields      $-xy \partial / \partial y$  and   $ -xy \partial / \partial x$ on $T$ under the natural embedding $T \to S$.
For $\delta_z$ it is a straightforward computation and we shall check only the fact that $\Ker \delta_z = \C [z]$.
Since the $\C [z]$ is contained in the kernel by Proposition \ref{nc.07.12.10} it suffices to show that $z$ cannot be presented as a nonlinear polynomial
$r(h)$ of another function $h$ on $S$. This follows immediately from the fact that the differential of $z$ vanishes at the set given
by $p'(x)x+(1-p(x)-q(y))= q'(y)y+(1-p(x)-q(y))=0$ which is finite. Hence we are done.

\end{proof}

\begin{theorem}\label{nc.12.16.40} The surface $S$ from Notation \ref{nc.12.16.10} has AVDP.
\end{theorem}

\begin{proof}
 As we mentioned before for every  $f \in \C [S ]$ there is $ \xi \in \AVF_\omega (S) $ for which
 $\iota_\xi \omega = {\rm d} f$.
Let $\cV$ be the subspace of $\C [S]$ consisting of all functions $f$ with $a_0 =0$ in Formula (4). Then
the map $f \to \xi$ induces an isomorphism $\Psi : \cV \to  \AVF_\omega (S) $. By Proposition \ref{nc.07.12.10} and Lemma \ref{nc.12.16.20}
up to constant factors $\Psi^{-1}$ sends
$\delta_z, \delta_y,$ and $\delta_x$ from Lemma \ref{nc.12.16.20}
to the functions $-z$,  $-y$ and $-x$ respectively  (which can be checked precisely by a direct computation).

Therefore for  $c(z)= -\sum_{i=1}^N ic_i z^{i-1}$  the  vector field $c(z) \delta_z$ is
complete, divergence-free, and $\Psi^{-1} (c(z) \delta_z) = \sum_{i=1}^N c_i z^{i}$ that is
the third nonconstant summand in the Formula (4).
The first and the second  nonconstant summands can be taken care of by vector fields of form $b (y) \delta_y$  and $a (x) \delta_x$.

Furthermore,
$$i_{[\delta_z, \delta_y]} \omega   =  {\rm d} (i_{\delta_z} ( i_{\delta_y} \omega)) = {\rm d}  (i_{\delta_z} ) {\rm d} y =
 {\rm d}  L_{\delta_z} (y)  = {\rm d} (1+yz) = {\rm d} (yz).$$
Thus $\Psi^{-1}$ sends the  Lie bracket $ [z^i \delta_z , y^j \delta_y ] \in \LieAO (S)$ to the monomial $z^{i+1} y^{j+1}$.
This shows that the last summand in Formula (4) is dual to an element from $\LieAO (S)$. The two remaining nonconstant summands can be treated similarly and thus
for any $f \in \cV$ one has  $\xi= \Psi (f) \in \LieAO (S)$
which yields the desired conclusion.

\end{proof}

\section{Appendix: algebraic volume forms on homogeneous spaces}

In this section we discuss some simple and perhaps known facts about algebraic volume forms.
If a smooth affine algebraic variety possesses such a form and does not admit
nonconstant invertible regular functions then the form is unique up to a constant factor.

Another well-known fact is that a linear algebraic group $G$ has a left-invariant algebraic volume form (which is simultaneously right-invariant
in the case of a reductive $G$)
but a homogeneous space $G/ R$ may not have a similar form (see Example \ref{nc.08.01.30} below).
The criterion for existence of such a form on $G/R$ is a straightforward analogue
of the criterion about the existence of an invariant Haar measure on a real homogeneous space in terms
of modular functions \footnote{
Recall that for a real Lie group $G_\R$ the modular function
$\Delta_{G_\R} : G_\R \to \R_+$ is given by $g \to |{\rm det} \, {\rm ad}_g |$ for $g \in G_R$ where ${\rm ad}_g$
is the adjoint action on the Lie algebra. If $R_\R$ is a closed Lie subgroup of $G_\R$ then the
homogeneous space $G_\R/R_\R$ has a $G_R$-invariant Haar measure iff $\Delta_{G_\R}|_R =\Delta_{R_\R}$.}.
In order to describe this criterion we need the following.

\begin{definition}\label{nc.08.01.10}  Let $G$  be a linear algebraic group, and $N$ be a subgroup
of $G$, and $H$ be a subgroup of the normalizer of $N$ in $G$.
Consider function $\tilde \Delta_{H,N} :  H \to \C^*$ that assigns to each $h \in H$ the determinant $ {\rm det} \, {\rm ad}_h|_{\ngoth} $  of the adjoint
action of $h$ on the Lie algebra $\ngoth$ of $N$ (in particular  $\tilde \Delta_{H,N}$ is a character of $H$).
We say that  $\tilde \Delta_{H,N}$ the sub-modular function of the pair $(H, N)$. In the case of $H=N=G$ we
call $\tilde \Delta_G := \tilde \Delta_{G,G}$ the sub-modular function of $G$.
\end{definition}

\begin{proposition}\label{nc.08.01.20} Let $G$ be a connected linear algebraic group, $\cR_u$ be the normal subgroup of $G$
associated with the unipotent radical of the Lie algebra of $G$, $H$ be a
maximal reductive subgroup of $G$, and $T$ be the connected identity component of the center of $H$.  Then
$\tilde \Delta_G \equiv 1$ iff $\tilde \Delta_{T, \cR_u } \equiv 1$. In particular, for every connected reductive group
its sub-modular function is the trivial character.

\end{proposition}

\begin{proof}
Let $S$ be a maximal semi-simple subgroup of $H$. The absence of nontrivial characters on $S$ and $\cR_u$ implies that
$\tilde \Delta_{S,G} \equiv 1$ and $\tilde \Delta_{\cR_u , G} \equiv 1$.  One can present
each $g \in G$ as $g=str$ where $s \in S, t \in T$, and $r \in \cR_u$.  Hence
$${\rm det} \, {\rm ad}_g=  {\rm det} \, {\rm ad}_s \cdot  {\rm det} \, {\rm ad}_t \cdot  {\rm det} \, {\rm ad}_r   = {\rm det} \, {\rm ad}_t,$$
i.e. $\tilde \Delta_G \equiv 1$ iff $\tilde \Delta_{T, G} \equiv 1$. Note   also that as a vector space the Lie algebra
of $G$ is the direct sum of Lie algebras of $H$ and $\cR_u$. That is, $\tilde \Delta_{T,G} = \tilde \Delta_{T,H} \tilde \Delta_{T,\cR_u}$.
Since $T$ is in the center of $H$ one has $ \Delta_{T, H} \equiv 1$, i.e. $\tilde \Delta_{T,G} = \tilde \Delta_{T,\cR_u}$
which implies the desired conclusion.
\end{proof}

Before formulating the criterion we need one more simple fact.

\begin{lemma}\label{nc.08.06.10}  Let $\rho : P \to X$ be a principal $R$-bundle where $R$ is a reductive group of dimension $m$.
Suppose that $\omega_R$ is an invariant volume form on $R$ and $\alpha$ is a (resp. closed; resp. exact)
$k$-form on $X$ where $0 \leq k \leq n:=\dim X$. Then there exist a  (resp. closed; resp. exact) $R$-invariant  $(m+k)$-form
$\alpha_P$ on $P$ such that for any open subset $U \subset X$, for which $\rho^{-1} (U)$ is naturally isomorphic to $U \times R$,
the restriction of $\alpha_P$ to $\rho^{-1} (U)$ coincides with $\alpha \times \omega_R$.
%In particular, the map $\alpha \to \alpha_P$ yields a homomorphism $\psi_k : H^k(X, \C ) \to H^{n+k} (P, \C)$.
\end{lemma}

\begin{proof}
Consider an open covering $\{ U_i \}$ of $X$ such that $\rho^{-1} (U_i)$ is naturally isomorphic to $U_i \times R$.
The structure of direct product enables us to consider an $R$-invariant form
$\omega_i =\alpha \times \omega_R$ on every $\rho^{-1} (U_i)$.
The transition isomorphism over $U_i \cap U_j$ is of form $(u, r) \to (u, g(u)r)$ where $u \in U_i \cap U_j$ and $r,g(u) \in R$.
Since $\omega_R$ is an invariant form we see that $\omega_i$ and $\omega_j$ agree on $\rho^{-1} (U_i \cap U_j)$ which
implies the desired conclusion.
\end{proof}

%\begin{notation}\label{nc.08.07.20} We denote such an $\alpha_P$ by $\alpha \tilde \wedge \omega_R$ because of the following reason.
%Form $\alpha$ induces naturally a $k$-form $\tilde \alpha$ on $P$ and there exists an $m$-form
%$\tilde \omega_R$ on $P$ whose restriction to each fiber of $\rho$ coincides with $\omega_R$ ($\tilde \omega_R$ is well-defined
%by the same reason $\alpha_P$ is). Now one can see that
%$\alpha \tilde \wedge \omega_R = \tilde \alpha \wedge \tilde \omega_R$.
%
%Note also that if $\omega_X$ is a volume form on $X$ then $\omega_X \tilde \wedge \omega_R$ is a volume form on $P$.
%%Using the Leray spectral sequence one can show that in the case of a simply connected $X$ the homomorphism $\psi_{n+m-1}$ is injective.
%
%\end{notation}

\begin{proposition}\label{nc.08.02.10} Let $X=G/R$ be a homogeneous space of left cosets
where $G$ is a linear algebraic group and $R$ is a closed reductive subgroup of $G$.
Then $\tilde \Delta_R \equiv \tilde \Delta_G |R$
if and only if there exists  an algebraic
volume form $\omega_X$ on $X$ invariant under the action of $G$
generated by left multiplication. In particular, in the case of connected reductive $G$ and $R$
such a volume form $\omega_X$ always exists (by Proposition \ref{nc.08.01.20}).

\end{proposition}

\begin{proof}
Choose a left-invariant volume
form $\omega$ on $G$ and left-invariant vector fields $\nu_1, \ldots , \nu_m$ on $G$
 generating the Lie algebra of $R$. 
%the coset $eR\simeq R$, where $e$ is the identity of $G$,  so that they
%generate basis of the tangent space at any point of this coset.
%Extend these vector fields to $G$ using left
%multiplication. Since $eR$ is a fiber of the natural projection $p : G \to X$ and left
%multiplication preserves the fibers of $p$ we see that the
%extended fields are tangent to all fibers of $p$.
These fields are tangent to all fibers of the natural projection $p : G \to X$.
Consider the left-invariant form $\omega_X=
\iota_{\nu_1} \circ \ldots \circ \iota_{\nu_m} (\omega)$. By construction it can be viewed
as a non-vanishing form on vectors from the pull-back of the tangent bundle $TX$ to $G$.
To see that it is actually a volume form on $X$ we have to show that it is invariant under multiplication by
any  element $r \in R$.  Such multiplication generates an automorphism of
$TG$ that sends vectors tangent (and, therefore,
transversal) to fibers of $p$ to similar
vectors. Hence it transforms $\omega_X$ into $f_r\omega_X$ where $f_r$ is an invertible function on $G$.
Since modulo $\C^*$ the group of invertible functions is a discrete set (more precisely, it is $H^1(G, \Z )$; e.g., see \cite{Fu})
and $f_e\equiv 1$ (where $e$ is the identity of $G$)
we see that $f_r$ is a nonzero constant for every $r$. Hence the map $r\to f_r$ yields  a homomorphism from $R$ into $\C^*$.
Since $\omega$ is left-invariant we see that ${r^{-1}} \circ \omega_X \circ r =f_r \omega_X$ (where $r^{-1} \circ \omega_X \circ r$ is the image of $\omega_X$
under conjugation by $r$).
Consider the last equality at $e \in G$ treating $\omega_X$ as the result of the evaluation of $\omega$ at the wedge
product $\mu=\nu_1 \wedge \ldots \wedge \nu_m$.  Note that conjugation by $r$ transforms $\omega$ at $e$ into $\tilde \Delta_G (r)\omega$ while
transforming $\mu$ into $(\tilde \Delta_R (r))^{-1}\mu$ (because $\mu$ is dual to $\omega_R$).  Hence $f_r (e) =\tilde \Delta_G (r)/\Delta_R (r)$.
Since $f_r(g)$ is independent from $g\in G$ we see that $\omega_X$ is invariant under right multiplication by elements of $R$ iff
$\tilde \Delta_R \equiv \tilde \Delta_G |_R$.

The other direction: by Lemma \ref{nc.08.06.10} in the presence of an algebraic volume form $\omega_X$ on $X$
consider  the volume form $\omega'$ on $G$ such that for every open $U \subset X$
with $p^{-1} (U)\simeq U \times R$ the restriction of $\omega'$ to $p^{-1}(U)$ coincides with $\omega_X \times \omega_R$.
Note that up to a constant factor $\iota_{\nu_1} \circ \ldots \circ \iota_{\nu_m} (\omega')$ coincides with $\omega_X$ since
$\iota_{\nu_1} \circ \ldots \circ \iota_{\nu_m} (\omega_R)$ is constant and
$\iota_{\nu_1} \circ \ldots \circ \iota_{\nu_m} (\omega_X \times \omega_R) =(\iota_{\nu_1} \circ \ldots \circ \iota_{\nu_m} (\omega_R)) \omega_X$.

Thus  we can suppose that
$\omega_X= \iota_{\nu_1} \circ \ldots \circ \iota_{\nu_m} (\omega')$. Note also that  $\omega' $
is left-invariant provided $\omega_X$ is left-invariant, i.e. $\omega'=\omega$. That is, the relation between the left-invariant form $\omega$ on $G$
and $\omega_X$ is the same as in the first part of the proof which yields the desired conclusion.

\end{proof}

\begin{example}\label{nc.08.01.30} Unlike in Proposition \ref{nc.08.01.20}  for a non-connected reductive group
the sub-modular function may differ from a trivial character.
Consider for instance a nontrivial extension $N$ of $\C^*$ by $\Z_2$, i.e. for
every $z \in \C^*$ and $g \in N \setminus \C^*$ we have $gzg^{-1} =z^{-1}$.
Then the value $\tilde \Delta_N$ on one connected component of $N$ is $-1$.
In particular, if one treats $N$ as the normalizer of a maximal torus in $SL_2$ then
$SL_2/N$ does not possess a left invariant volume form  by Proposition \ref{nc.08.02.10}. Furthermore,
$SL_2/N$ does not admit nonconstant invertible regular functions and therefore any two algebraic volume forms
must be proportional which implies the absence of algebraic volume forms of $SL_2/N$.
\end{example}

It is worth mentioning the following cohomological interpretation of volume forms.

\begin{proposition}\label{nc.08.08.10}
Let $X$ be a homogeneous space of a connected reductive group $G$
and $\omega$ be an $G$-invariant volume form on $X$. Suppose that $\dim X =n$ and $H^n(X) \ne 0$ (i.e. $H^n(X, \C ) =\C$).
Then De Rham homomorphism sends $\omega$ to a generator of $H^n (X, \C)$.
\end{proposition}

\begin{proof} Suppose that under De Rham isomorphism a closed form $f \omega$ corresponds to a generator  $\alpha$ of $H^n(X, \Z)$ and
that $K=G_\R$ is a maximal compact subgroup of $G$ (i.e.
$G$ is the complexification $K_\C$ of $K$). Since the $K$-action on $X$ induces the identical automorphism
of $H^n (G, \Z )$ we see that $(f \circ k) \omega$ corresponds again to $\alpha$ where $f\circ k$ is the image of $f$ under the action
of $k \in K$. Hence $g \omega$ corresponds to $\alpha$ where $g = \int_K (f \circ k) \mu$ and $\mu$ is the invariant Haar measure on $K$.
Since $g$ is $K$-invariant it must be a constant function on $X$ which implies the desired conclusion.

\end{proof}

\begin{corollary}\label{nc.11.30.10}
Let an
affine algebraic manifold $X$ without nonconstant regular invertible functions possess an algebraic volume form $\omega$
such that under De Rham homomorphism $\omega$ corresponds to the zero element of $H^n(X, \C ) =\C$.
Then  $X$ cannot be a homogeneous space of any   reductive group.

\end{corollary}

\begin{remark}\label{nc.11.2.10}  It was shown in \cite{DuFi} that any variety $X_{m,1}= \{ x^mv-yu=1 \} \subset \C^4_{x,y,u,v}$
with $m\geq 2$ is diffeomorphic (as a real manifold) but not isomorphic to $X_{1,1} \simeq SL_2$ \footnote{Such varieties $X_{m,1}$
with $m \geq 2$ are examples of ``complex exotic affine 3-spheres" since $SL_2$ is isomorphic to $\{ x^2 +y^2 +u^2 +v^2 =1 \} \subset \C^4$.}
because the unique (up to a constant factor) volume form $\omega_m =x^{-m} {\rm d} x \wedge {\rm d} y \wedge {\rm d} u $ on $X_{m,1}$ is exact
($\omega_m = {\rm d} \tau$ where $\tau =\frac {{\rm d} y \wedge{\rm d} u}{(1-m)x^{m-1}}$) .
Corollary \ref{nc.11.30.10} enables us to tell now more:
$X_{m,1}$ is not  isomorphic to a homogeneous space of a reductive group.

\end{remark}

\providecommand{\bysame}{\leavevmode\hboxto3em{\hrulefill}\thinspace}

\end{document}